\documentclass[11pt,letterpaper]{amsart}
\usepackage[margin=1.25in]{geometry}
\usepackage[utf8]{inputenc}
\usepackage[T1]{fontenc}
\usepackage{lmodern}

\usepackage{amsmath,amssymb,amsthm,mathtools}
\usepackage{mathrsfs}
\usepackage{enumitem}
\usepackage{tikz-cd}

\usepackage[
  colorlinks=true,
  linkcolor=blue,
  citecolor=blue,
  urlcolor=blue
]{hyperref}

\usetikzlibrary{
    arrows.meta,
    positioning,
    calc
}
\newtheorem{theorem}{Theorem}
\newtheorem{lemma}[theorem]{Lemma}
\newtheorem{conjecture}[theorem]{Conjecture}
\newtheorem{proposition}[theorem]{Proposition}
\newtheorem{note}[theorem]{Note}
\newtheorem{claim}[theorem]{Claim}
\newtheorem{remark}[theorem]{Remark}
\newtheorem{corollary}[theorem]{Corollary}
\newtheorem{example}[theorem]{Example}

\newcommand{\cO}{\mathcal O}
\newcommand{\cI}{\mathcal I}
\newcommand{\cE}{\mathscr E}
\newcommand{\cF}{\mathscr F}

\newcommand{\Ext}{\operatorname{Ext}}

\newcommand{\Ann}{\operatorname{Ann}}

\newcommand{\Supp}{\operatorname{Supp}}
\newcommand{\im}{\operatorname{im}}
\newcommand{\rk}{\operatorname{rk}}

\title[K\"ahler fourfolds without divisors and surfaces]{Classification of Smooth Minimal K\"ahler Fourfolds Without Effective Divisors and Surfaces}
\author{Pisya Vikash}
\date{}
\begin{document}
\flushbottom
\begin{abstract}
We prove the following two statements: 
\begin{enumerate}
    \item If $X$ is a compact K\"ahler fourfold with pseudo-effective canonical bundle and no subvarieties of codimension one or two, then $K_X$ is a torsion line bundle.
    \item If $X$ is a compact K\"ahler fourfold with no subvarieties of codimension one or two, then $X$ is not uniruled.
\end{enumerate}
 Using the Beauville-Bogomolov decomposition and a characterization of Uniruled manifolds by Wenhao Ou \cite{OuUniruledKahler}, we get that if $X$ is a compact K\"ahler fourfold with no subvarieties of codimension one or two, then either $X$ is a torus or an irreducible holomorphic symplectic manifold.
\end{abstract}
\maketitle

\begin{center}
\emph{Dedicated to Professor Yuriy G. Zarhin, with deep gratitude for his guidance and inspiration.}
\end{center}



\section{Introduction}

The birational classification of compact K\"ahler manifolds is closely related to the existence of minimal models and the abundance conjecture. Substantial progress in dimension three was made by Campana, Peternell, H\"oring, Omprokash Das, Wenhao Ou \cite{CP97,Pet98,Pet01,HP16,CampanaHoringPeternell2016,CHPErratum,DasOu24,DasOu26}. Here we consider compact K\"ahler fourfold $X$ containing no divisors or surfaces. MMP together with abundance conjectures predict that $X$ is a torus or irreducible holomorphic symplectic manifolds when the canonical bundle $K_X$ is pseudo-effective, but those tools are not available in dimension four. By an irreducible holomorphic symplectic manifold we mean a simply connected compact K\"ahler manifold $X$ such that $H^{2,0}(X)=\mathbb{C}\sigma$, where $\sigma$ is a nowhere-vanishing holomorphic symplectic $2$-form.

The classification problem discussed in the paper is also related to the theory of simple K\"ahler manifolds. A manifold $X$ of dimension at least two is called simple if a very general point lies on no proper positive-dimensional analytic subvariety. Such manifolds are studied in  \cite{campana2006isotrivialite,MoosaPillay2008,CampanaDemaillyVerbitsky2014,Campana2026Bogomolov}. The following conjecture was proposed by Campana, Demailly, Moosa, Pillay and Verbitsky.

\begin{conjecture}\label{conj:simple-kahler}
Let $X$ be a simple compact K\"ahler manifold of dimension at least two. Then either $X$ is bimeromorphic to $T/G$, where $T$ is a complex torus and $G$ is a finite subgroup of $\operatorname{Aut}(T)$, or $H^0(X,\Omega_X^2)$ is generated by a generically symplectic form $\sigma$. In the latter case, $\dim X=2m$ and $\sigma^{\wedge m}\neq0$. In particular, one expects $\kappa(X)=0$ in both cases.
\end{conjecture}
The following formulation concerns the stronger condition that every proper analytic subvariety is zero-dimensional.
\begin{conjecture}\label{conj:no-subvarieties}
Let $X$ be a compact K\"ahler manifold of dimension at least two containing no proper positive-dimensional analytic subvarieties. Then $X$ is either a finite \'etale quotient of a complex torus or an irreducible holomorphic symplectic manifold. In particular, $K_X$ is torsion.
\end{conjecture}

\subsection{Main Results}
The main results that we prove are the following.

\begin{theorem}\label{thm:thmmain}
    Let $X$ be a compact K\"ahler manifold of dimension $4$, with pseudo-effective canonical bundle. Assume that, $X$ has no subvarieties of codimension $1$ and $2$. Then, $K_X$ is torsion.
\end{theorem}
\begin{theorem}\label{thm:thmuniruled}
    Let $X$ be a smooth compact K\"ahler fourfold containing no analytic subvarieties of codimension one or two. Then $X$ is not uniruled.
\end{theorem}
It would be interesting to know whether the assumption excluding codimension
$2$ subvarieties can be removed from Theorem \ref{thm:thmmain} because this assumption is used exactly at one place after we reduce it to nef case. The following simple examples illustrate the hypotheses.

\begin{example}
Let $T$ be a four-dimensional complex torus with no 
proper analytic subvarieties. They exist; see
\cite{bandman2023simple}. Then $T$ satisfies the assumptions of
Theorem~\ref{thm:thmmain}, and its canonical bundle is trivial.
\end{example}

\begin{example}
The assumption that $X$ contains no divisors is essential. Let $B$ be a
three-dimensional complex torus with no positive-dimensional proper
analytic subvarieties, and let $C$ be a smooth curve of genus two.
The product
\[
    X:=B\times C
\]
is a compact K\"ahler fourfold with
\[
    K_X\simeq \operatorname{pr}_C^*K_C,
\]
which is pseudo-effective but not torsion, here $\operatorname{pr}$ is projection of $X$ to $C$. Obviously, $X$ contains no
surfaces. On the other hand, $B\times\{p\}$ is a divisor in $X$ for every $p\in C$.
\end{example}

\begin{example}
The absence of surfaces does not follow from the absence of divisors.
Let $S_1$ and $S_2$ be K3 surfaces with trivial Picard groups, as is the
case for very general K3 surfaces. Then
\[
    X:=S_1\times S_2
\]
also has a trivial Picard group and hence contains no divisors. Thus, the condition excluding codimension-two subvarieties is a genuinely additional hypothesis.
\end{example}
To prove Theorem \ref{thm:thmmain} we break the proof into important lemmas.
\begin{lemma}\label{thm:main} Let $X$ be a connected compact K\"ahler fourfold and set $K:=K_X$. Assume that:
\begin{enumerate}[label=\textup{(\arabic*)}]
\item $X$ has no irreducible proper analytic subvarieties of codimension $1$ or $2$;
\item $K$ is pseudo-effective;
\item $h^0(X,\Omega_X^3)\geq 2$.
\end{enumerate}
Then $K_X \text{ is torsion}$.
\end{lemma}
This lemma is useful in its own right: the examples in Section~\ref{ex:gen} show that
its hypotheses occur naturally. The Lemma \ref{thm:main} is enough to prove Conjecture \ref{conj:no-subvarieties} in dimension four. The main new idea of the proof is recursively constructing extension classes. We next study the restrictions imposed on the irregularity of $X$ by the assumption that $X$ contains no irreducible proper analytic subvarieties of complex codimension one and two.
\begin{lemma}\label{thm:0or4}
    Let $X$ be a compact K\"ahler manifold of dimension $4$, with pseudo-effective canonical bundle. Assume that, $X$ has no subvarieties of codimension $1$ and $2$. Then, $q(X)= 0 \text{ or } 4$.
\end{lemma}
\begin{lemma}\label{thm:mainnef}
    Let $X$ be a compact K\"ahler fourfold with $K_X$ nef. Assume that, $X$ has no subvarieties of codimension $1$ and $2$. Then, $K_X$ is torsion.   
\end{lemma}
\begin{lemma}\label{lem:pseff-canonical-nef-no-divisors-surfaces}
Let $X$ be a compact K\"ahler fourfold. Assume that $X$ contains no
irreducible analytic subvarieties of codimension $1$ or $2$. If $K_X$ is
pseudo-effective, then $K_X$ is nef.
\end{lemma}

A recent preprint by Wenhao Ou  establishes the following statement. 
\begin{theorem}[\cite{OuUniruledKahler}]
    A compact K\"ahler manifold is uniruled if and only if $K_X$ is not pseudo-effective.
\end{theorem}
Using this we can obtain the following Corollaries. 

\begin{corollary}\label{cor:kahler-deforms-no-subvarieties}
 Let $X$ be a compact K\"ahler fourfold. Suppose that $X$ is deformation equivalent to a compact K\"ahler fourfold $Y$ containing no proper positive-dimensional analytic subvarieties. Then $X$ is either a complex torus or an irreducible holomorphic symplectic fourfold.
\end{corollary}

\begin{corollary}\label{cor:final}
    Let $X$ be a compact K\"ahler fourfold with no subvarieties of codimensions one and two, then $X$ is biholomorphic to a torus or irreducible holomorphic symplectic manifold.
\end{corollary}

At this point, it is obvious to ask the following question : If $X$ is a K\"ahler fourfold with no subvarieties of codimensions one and two, does that mean no curves ? The answer to this question is "no", but the natural strategy fails. Namely, trying to construct a rank two vector bundle $E$ on a torus $T$ and trying to prove that $\mathbb P(E)$ has no subvarieties of codimension one and two fails. Although M. Toma constructs a rank two vector bundle $E$ on a torus of dimension two with no subvarieties such that $\mathbb P(E)$ has no subvarieties of codimension one in \cite{Toma89}. This fails in dimension four and we show that no such construction exists namely we prove the following statement. 

\begin{proposition}\label{lem:p1-bundle-divisor}
    Let $T$ be a simple complex torus of dimension three. Then every holomorphic $\mathbb P^1$-bundle $p:Y\to T$ contains an effective divisor.
\end{proposition}

The "no" to the above question comes from a different construction. In section \ref{ex:gen}, we construct an elliptic fibration over a torus of dimension three with no subvarieties such that the total space has no subvarieties of codimensions one and two and is a torus. The idea is the following: We start with a three dimensional torus and build an torus extension by elliptic curve and choose a general complex structure so that it's Neron-severi is zero. Now observe that it has curves, no codimension one subvarieties because Neron-severi is zero, and no surfaces because we started with a torus with no subvarieties.

\subsection{Idea of the proof}
Let us first describe the proof idea of Theorem \ref{thm:thmmain}. Since the manifold lacks geometry, many usual methods fail here; however, such manifolds naturally admit a holomorphic two form; therefore, we study those holomorphic two forms; in \cite{HoringPeternellRadloff2013} H\"oring-Peternell-Radloff also study such holomorphic two forms, assuming further that $X$ admits a covering by $\mathbb C^4$, but the way we study them is different from what is done in \cite{HoringPeternellRadloff2013}. Suppose $K_X$ is pseudo-effective but not nef then we obtain a rational curve $C$ such that $c_1(K_X).C<0$ due to a result by Cao-Höring \cite[Corollary~1.4]{CaoHoering2020}. Now add a K\"ahler class $\omega$ to get  $\alpha= (c_1(K_X) + \epsilon\omega)$ such that $\alpha.C<0$. Demailly Regularization \cite[Theorem~1.1]{Demailly1992} gives a K\"ahler current in $\alpha$, with analytic singularities. Horikawa’s deformation estimate \cite{Horikawa1973} shows that such a rational curve moves in a positive dimensional family. Finally, we get a contradiction by showing all those rational curves live in the singular locus of this K\"ahler current. This is basically Lemma \ref{lem:pseff-canonical-nef-no-divisors-surfaces}.
\\
Assume $K_X$ is not torsion. Miyaoka–Yau inequality we show that if $\chi(X,\mathcal O_X)>0$ then $\chi(X,K_X^m)>0$ and get a contradiction using Proposition \ref{prop:propleq1}. Hence, $\chi(X,\mathcal{O}_X)\leq 0$ then due to the surjectivity of Albanese map we have $h^{1,0} \neq 1,2$. Using standard facts about fibrations by curve we exclude the the case $h^{1,0}=3$; Lemma \ref{thm:0or4}. Now negativity of Euler characteristic and existence of holomorphic $2$ form gives either $K_X$ is torsion or has at least $2$ holomorphic three forms. If $K_X$ is not torsion, we use the holomorphic forms to construct an exact sequence
\[
0\to K_X^{-1} \to \mathcal{E} \to I_{Z}\otimes K_X \to 0 .
\]
Where, $\mathcal{E}$ is the saturation of the image of a holomorphic two form and $I_Z$ some ideal in $\mathcal{O}_X$. We then prove that this is not split using a theorem of Pereira-Rousseau-Touzet \cite{PRT}. Therefore, defines a non zero extension class $\operatorname{Ext}^1(I_Z, K_X^{-2})$. Now we use the no surfaces assumption  to say it actually defines a class in $H^1(K_X^{-2})$, then use Serre duality and Hard-Lefschetz to get a non zero class in $H^0(\Omega^1\otimes K_X^2)$. Then show that this process can be done recursively and obtain a non zero class in $H^0(\Omega^1\otimes K_X^m)$, for infinitely many $m$. This gives a contradiction due to \cite[Proposition 2.6]{AH22}.

Now, coming to Theorem \ref{thm:thmuniruled}, it is easy to conclude $X$ is a $\mathbb P^1$ bundle over a simple torus $T$ using \cite[Theorem~2.6]{Campana04} and \cite{DasOu26}. After that The main architecture of the argument can be described using the following implications.
\[
\mathbb P^1\text{-bundle}
\xrightarrow{\text{finite \'etale base change}}
Y\simeq\mathbb P(E),\qquad \det E\simeq\mathcal O_{T'},
\]
here $E$ is a rank 2 vector bundle on some finite cover $T'$ of $T$. After that, we prove the following chain of implications.
\[
E\text{ has no rank-one coherent subsheaves}
\Longrightarrow
E\text{ is stable for all K\"ahler class}
\Longrightarrow
c_2(E)=0.
\]
The last implication uses Bogomolov-L\"ubke inequality with Lemma~\ref{lem:torus-positive-class}.
Then apply Uhlenbeck and Yau to get the contradiction \cite{UY86}. This is not a new architecture. For example, \cite{HoringPeternellRadloff2013} uses the same architecture. More recently, F. Campana used the same architecture in \cite{Campana2026Bogomolov}, but the proof of each implication is different from they do. The entire proof can be explained using the following diagram.
\begin{note}
During the preparation of this article, F. Campana posted a preprint considering the  four-dimensional no-subvariety case in \cite{Campana2026Bogomolov}. The present work was developed independently and allows curves; Section~\ref{ex:gen} gives examples covered by this broader statement. Method is also different from them and more elementary.
\end{note}

\resizebox{0.82\textwidth}{!}{%
\begin{tikzpicture}[
    font=\scriptsize,
    node distance=5mm and 8mm,
    box/.style={
        draw=blue!55!black,
        rounded corners=2pt,
        line width=0.6pt,
        align=center,
        inner sep=4pt,
        fill=white
    },
    mainbox/.style={
        box,
        text width=7cm
    },
    mediumbox/.style={
        box,
        text width=4.3cm
    },
    smallbox/.style={
        box,
        text width=2.8cm
    },
    casebox/.style={
        box,
        text width=3.5cm
    },
    arrow/.style={
        -{Latex[length=2mm,width=1.3mm]},
        line width=0.6pt
    }
]

\node[font=\bfseries\large] (title)
{Proof structure of the theorem};

\node[mainbox, below=5mm of title, align=left] (assumptions)
{
    Assume $X$ is a compact K\"ahler fourfold with:
    \begin{itemize}
        \setlength{\itemsep}{0pt}
        \setlength{\topsep}{1pt}
        \item no divisors,
        \item no surfaces.
    \end{itemize}
};

\node[mainbox, below=5mm of assumptions] (nonuniruled)
{
 \textbf{Theorem \ref{thm:thmuniruled}}\\
    $X$ is not uniruled. Hence $K_X$ is pseudo-effective.
};

\node[smallbox, below=5mm of nonuniruled] (nef)
{
\textbf{Step 1:} $K_X$ is nef,\\
Lemma \ref{lem:pseff-canonical-nef-no-divisors-surfaces}
};
\node[mediumbox, below=5mm of nef] (irregularity)
{
Irregularity dichotomy:\\
$q(X)=0$ or $q(X)=4$, Lemma \ref{thm:0or4}
};
\node[smallbox, below left=7mm and 22mm of irregularity] (qzero)
{$q(X)=0$};
\node[smallbox, below right=7mm and 22mm of irregularity] (qfour)
{$q(X)=4$};
\node[mediumbox, below=5mm of qfour] (torus)
{$X$ is a complex torus};
\node[mediumbox, below=5mm of qzero] (split)
{Split according to $\chi(X,\mathcal O_X)$};
\node[casebox, below left=7mm and 9mm of split] (caseone)
{
\textbf{Case 1:}\\
$\chi(X,\mathcal O_X)\leq 0$
};
\node[casebox, below right=7mm and 9mm of split] (casetwo)
{
\textbf{Case 2:}\\
$\chi(X,\mathcal O_X)>0$
};
\node[casebox, below=5mm of caseone] (threeforms)
{
$h^0(X,\Omega_X^3)\geq 2$\\
Two holomorphic $3$-forms
};
\node[casebox, below=5mm of threeforms] (extension)
{
Apply the $3$-form and\\
extension-class argument, Lemma \ref{thm:main}
};
\node[casebox, below=5mm of extension] (torsionone)
{
\textbf{Conclusion:}\\
$K_X$ is torsion
};
\node[casebox, below=5mm of casetwo] (miyaoka)
{
Miyaoka-Yau inequality\\
$+$ Hirzebruch-Riemann-Roch
};
\node[casebox, below=5mm of miyaoka] (positivity)
{
Positivity of $\chi(X,K_X^m)$
};
\node[casebox, below=5mm of positivity] (torsiontwo)
{
\textbf{Conclusion:}\\
$K_X$ is torsion; Lemma \ref{thm:mainnef}
};
\node[
mainbox,
below=9mm of $(torsionone.south)!0.5!(torsiontwo.south)$
] (final)
{
\textbf{Final conclusion:}\\
$X$ is either torus or Hyperkahler.
};
\draw[arrow] (assumptions) -- (nonuniruled);
\draw[arrow] (nonuniruled) -- (nef);
\draw[arrow] (nef) -- (irregularity);
\draw[arrow] (irregularity.south west) -- (qzero.north);
\draw[arrow] (irregularity.south east) -- (qfour.north);
\draw[arrow] (qzero) -- (split);
\draw[arrow] (qfour) -- (torus);
\draw[arrow] (split.south west) -- (caseone.north);
\draw[arrow] (split.south east) -- (casetwo.north);
\draw[arrow] (caseone) -- (threeforms);
\draw[arrow] (threeforms) -- (extension);
\draw[arrow] (extension) -- (torsionone);
\draw[arrow] (casetwo) -- (miyaoka);
\draw[arrow] (miyaoka) -- (positivity);
\draw[arrow] (positivity) -- (torsiontwo);

\draw[arrow]
    (torsionone.south) -- ++(0,-3mm) -| (final.north west);
\draw[arrow]
    (torsiontwo.south) -- (final.north);
\draw[arrow]
    (torus.south) -- ++(0,-5mm) |- (final.north east);

\end{tikzpicture}%
}

\subsection{Notation and conventions}
\begin{description}
    \item[$T_X$]  Holomorphic tangent bundle of $X$.

    \item[$\Omega_X^p$]  Sheaf of holomorphic $p$-forms on $X$.

    \item[$K_X$] Canonical bundle $\det\Omega_X^1$, often written as $K$.

    \item[$q(X)$]  Irregularity
    $h^1(X,\mathcal O_X)=h^0(X,\Omega_X^1)$ of a compact K\"ahler manifold.

    \item[$\chi(X,\mathcal O_X)$] Holomorphic Euler characteristic of $X$.

    \item[$a(X)$] Algebraic dimension of $X$.

    \item[$\operatorname{Alb}(X)$]  Albanese torus of $X$.

    \item[$\alpha_X$]  Albanese map
    $X\to\operatorname{Alb}(X)$.

    \item[$\mathscr F^\vee$] 
    $\mathcal Hom_{\mathcal O_X}(\mathscr F,\mathcal O_X)$.

    \item[$\det\mathscr F$] Reflexive determinant
    $(\bigwedge^r\mathscr F)^{\vee\vee}$ of a torsion-free sheaf
    $\mathscr F$ of rank $r$.

    \item[$\operatorname{rk}\mathscr F$] Rank of $\mathscr F$.

    \item[$\operatorname{Supp}\mathscr F$] Support of $\mathscr F$.

    \item[$\operatorname{Ann}(\mathscr F)$] Annihilator of $\mathscr F$.

    \item[$\operatorname{im}\phi,\ \operatorname{coker}\phi$]
    The image and cokernel of a morphism $\phi$.

    \item[$\operatorname{Ext}^i_X(-,-)$] Global Ext groups.

    \item[$\mathcal Ext^i_X(-,-)$] Ext Sheaf.

    \item[$I_Z$] Ideal sheaf of a closed analytic subspace $Z\subset X$.

    \item[$|Z|$] Support of $Z$.

    \item[$\operatorname{codim}_X|Z|\geq c$]
    Every irreducible component of $|Z|$ has codimension at least $c$ in $X$.

    \item[$\Theta_h$] Curvature current of a possibly singular
    Hermitian metric $h$ on a line bundle.

    \item[$T=V/\Lambda$] A complex torus, with $V$ a complex vector space
    and $\Lambda\subset V$ a lattice.

    \item[$T/G$] Analytic quotient by a finite group $G$
    acting holomorphically on $T$.

    \item[$L_g$] Linear part on $V$ of the automorphism
    induced by $g\in G$.
\end{description}

\section{Preliminaries}
\subsection{Reflexive sheaves} 
Most of the results in this section can be found in \cite{OSS,Hartshorne1980}. Let $X$ be a complex manifold and let $\mathscr F$ be a coherent sheaf on $X$. We write
\[
\mathscr F^\vee:=\mathcal Hom_{\mathcal O_X}(\mathscr F,\mathcal O_X).
\]
There is a natural evaluation morphism
\[
\mathscr F\longrightarrow \mathscr F^{\vee\vee}.
\]
The sheaf $\mathscr F$ is called reflexive if this morphism is an isomorphism.
If $\mathscr F$ is torsion-free of rank $r$, its determinant is understood in
the reflexive sense:
\[
\det\mathscr F:=\left(\bigwedge^r\mathscr F\right)^{\vee\vee}.
\]
When $\mathscr F$ is locally free, this agrees with the usual determinant. We shall use the following standard facts. On a smooth complex manifold,
rank-one reflexive coherent sheaves are locally free, hence line bundles
\cite[Ch.~II, Lemma~1.1.15]{OSS}. Reflexive sheaves, their morphisms, and their
sections are determined by restriction to the complement of an analytic subset
of codimension at least two. More precisely, if $A\subset X$ is an analytic
subset of codimension at least two and
$j:X\setminus A\hookrightarrow X$ is the inclusion, then every reflexive
coherent sheaf $\mathscr F$ on $X$ satisfies
\[
\mathscr F\simeq j_*(\mathscr F|_{X\setminus A}),
\]
and hence
\[
H^0(X,\mathscr F)=H^0(X\setminus A,\mathscr F|_{X\setminus A});
\]
See \cite[Proposition 1.6]{Hartshorne1980}. A coherent subsheaf $\mathscr G\subset \mathscr F$ is called saturated in
$\mathscr F$ if the quotient $\mathscr F/\mathscr G$ is torsion-free.
Equivalently, $\mathscr G$ is maximal among subsheaves of $\mathscr F$
which agree with $\mathscr G$ at the generic point. Thus, if
\[
0\to \mathscr K\to \mathscr E\to \mathscr Q
\]
is exact, with $\mathscr E$ reflexive and $\mathscr Q$ torsion-free, then
$\mathscr K$ is a saturated subsheaf of $\mathscr E$, hence reflexive
\cite[Ch.~II, Lemma~1.1.16]{OSS}. Finally, if $I_Z\subset \mathcal O_X$ is a
coherent ideal sheaf whose cosupport has codimension at least two, then
\[
I_Z^{\vee\vee}\simeq\mathcal O_X \qquad \text{ and } \qquad \det(I_Z\otimes L)\simeq L
\]
for every line bundle $L$.

\begin{lemma}\label{lem:no-meromorphic}
Let $X$ be a compact complex manifold with no codimension-one subvarieties. Then every meromorphic function on $X$ is constant.
\end{lemma}
\begin{proof}
Every meromorphic function on $X$ has no polar divisor, hence is holomorphic by normality and constant by compactness.
\end{proof}

\begin{lemma}\label{lem:line-sections}
Let $X$ be a compact complex manifold with no codimension-one subvarieties.
If $L$ is a line bundle on $X$ and
$
0\neq s\in H^0(X,L),
$
then $s$ is nowhere vanishing. Hence $L\simeq \mathcal O_X$. In particular, if
$
H^0(X,K_X^r)\neq 0
$
for some $r\in \mathbb Z\setminus\{0\}$, then $K_X$ is torsion.
\end{lemma}

\begin{proof}
 $s$ trivializes $L$; which proves the final assertion.
\end{proof}

\subsection{Multiplier ideals and hard Lefschetz}

Most of the results in this section can be found in \cite{DemaillyAnalyticMethods,DPS01}. Let $L$ be a holomorphic line bundle on a complex manifold $X$. A singular
Hermitian metric $h$ on $L$ is given locally by
\[
|e|_h^2=e^{-2\varphi}
\]
with respect to a local holomorphic frame $e$ of $L$, where $\varphi$ is
an $L^1_{\mathrm{loc}}$ function. Its curvature current is locally
\[
\Theta_h(L)=2\partial\bar\partial \varphi .
\]
We say that $(L,h)$ has semipositive curvature if the local weights
$\varphi$ are plurisubharmonic, equivalently if $\Theta_h(L)$ is a closed
positive $(1,1)$-current. The multiplier ideal sheaf associated to $h$ is the coherent ideal sheaf
$
\mathcal I(h)\subset \mathcal O_X
$
defined locally by
\[
\mathcal I(h)(U)
=
\left\{
f\in \mathcal O_X(U)\;:\;
|f|^2 e^{-2\varphi}\text{ is locally integrable on }U
\right\}.
\]
The coherence of $\mathcal I(h)$ is Nadel's coherence theorem; see
\cite[Theorem~5.7]{DemaillyAnalyticMethods} when $\phi$ is plurisubharmonic. A line bundle $L$ on a compact
complex manifold is called pseudo-effective if it admits a singular Hermitian
metric with semipositive curvature current. A line bundle $L$ on a compact K\"ahler manifold $(X,\omega)$, where $\omega$ is K\"ahler form, is called nef if
\[
c_1(L)\in \overline{\mathcal K}_X,
\]
where $\mathcal K_X\subset H^{1,1}(X,\mathbb R)$ denotes the K\"ahler cone.
Equivalently, for every $\varepsilon>0$, there exists a smooth Hermitian
metric $h_\varepsilon$ on $L$ such that
\[
\sqrt{-1}\Theta_{h_\varepsilon}(L)\geq -\varepsilon\omega
\]
for every K\"ahler form $\omega$ on $X$. In particular, every nef line bundle on a compact K\"ahler manifold is pseudo-effective. We shall use the following hard Lefschetz theorem with multiplier ideals due to
Demailly-Peternell-Schneider.
\begin{theorem}[Demailly-Peternell-Schneider\cite{DPS01}]\label{thm:DPS-HL}
Let $X$ be a compact K\"ahler manifold of dimension $n$, let $L$ be a
pseudo-effective line bundle, and let $h$ be a singular Hermitian metric on
$L$ with semipositive curvature current. Then, for every $q\geq 0$,
cup-product with a K\"ahler class gives a surjective map
\[
H^0\bigl(X,\Omega_X^{n-q}\otimes L\otimes \mathcal I(h)\bigr)
\xrightarrow{\wedge \omega^q}
H^q\bigl(X,K_X\otimes L\otimes \mathcal I(h)\bigr).
\]
\end{theorem}
\subsection{Foliations}
The results in this section are proved in \cite{DemaillyFrobenius,PRT,TouzetConormal}.
Let $X$ be a complex manifold. A possibly singular holomorphic foliation on
$X$ is a saturated coherent subsheaf
$
\mathcal F\subset T_X
$
which is closed under the Lie bracket. Its rank is the general rank of
$\mathcal F$, and its codimension is
$
\operatorname{codim}\mathcal F=\dim X-\operatorname{rk}\mathcal F.
$
The normal and conormal sheaves are
\[
N_{\mathcal F}:=(T_X/\mathcal F)^{**},
\qquad
N^*_{\mathcal F}:=N_{\mathcal F}^{\vee}.
\]

In this paper we only use codimension-one foliations. In this case
$N^*_{\mathcal F}$ is a rank-one reflexive subsheaf of $\Omega_X^1$, hence
a line bundle on a smooth complex manifold. Equivalently, a codimension-one
foliation is given by a saturated rank-one subsheaf
$
N^*_{\mathcal F}\subset \Omega_X^1
$
satisfying the Frobenius integrability condition. Locally, if
$N^*_{\mathcal F}$ is generated by a holomorphic one-form $\alpha$, this
condition is
$
\alpha\wedge d\alpha=0.
$
Moreover,
\[
\mathcal F=\ker\bigl(T_X\to (N^*_{\mathcal F})^\vee\bigr),
\qquad
N^*_{\mathcal F}=\operatorname{Ann}(\mathcal F)\subset \Omega_X^1.
\]
We shall use the following theorem of Demailly.

\begin{theorem}[\cite{DemaillyFrobenius}]\label{thm:demailly-frobenius}
Let $X$ be a compact K\"ahler manifold, let $L$ be a pseudo-effective line
bundle on $X$, and let
\[
0\neq \theta\in H^0(X,\Omega_X^1\otimes L^{-1}).
\]
Then the coherent subsheaf
\[
S_\theta:=\{\xi\in T_X:\iota_\xi\theta=0\}
\]
is integrable. Equivalently, the saturated rank-one subsheaf of
$\Omega_X^1$ generated by $\theta$ defines a possibly singular
codimension-one holomorphic foliation.
\end{theorem}
For a line bundle $L$ on a compact complex manifold, its Iitaka dimension is
\[
\kappa(L)
:=
\begin{cases}
\max_{m>0}\dim \Phi_{|L^m|}(X), & \text{if } H^0(X,L^m)\neq 0
\text{ for some }m>0,\\
-\infty, & \text{otherwise.}
\end{cases}
\]
Here $\Phi_{|L^m|}$ denotes the meromorphic map defined by the complete linear
system $|L^m|$. If $L$ is pseudo-effective, its numerical dimension is denoted by $\nu(L)$.
For nef $L$, this can be characterized by
\[
\nu(L)=\max\{k\geq 0: c_1(L)^k\neq 0\}.
\]
More generally, for pseudo-effective line bundles one uses the numerical
dimension of the positive part in the divisorial Zariski decomposition. Following \cite{PRT} We call
a pseudo-effective line bundle $L$ abundant if
\[
\kappa(L)=\nu(L).
\]
For conormal bundles of codimension-one foliations, Touzet's structure theorem \cite{TouzetConormal}
implies that the only possible values are
\[
\nu(N^*_{\mathcal F})=0,\ \kappa(N^*_{\mathcal F})=0,
\]
or
\[
\nu(N^*_{\mathcal F})=1,\ \kappa(N^*_{\mathcal F})=1,
\]
or
\[
\nu(N^*_{\mathcal F})=1,\ \kappa(N^*_{\mathcal F})=-\infty.
\]
Thus the last case is precisely the non-abundant case. We shall use the following theorem of Pereira-Rousseau-Touzet.

\begin{theorem}[Pereira-Rousseau-Touzet]\label{thm:PRT}
Let $X$ be a compact K\"ahler manifold and let $\mathcal F$ be a
codimension-one holomorphic foliation on $X$. Assume that the conormal bundle
$N^*_{\mathcal F}$ is pseudo-effective and not abundant. Then there exists an
integer $p\geq 2$ and an invertible subsheaf
\[
\mathcal M\subset \Omega_X^p \qquad \text{ such that } \qquad\kappa(\mathcal M)=p.
\]
\end{theorem}
\section{Proof of Main Theorem}
\subsection{Recursive extension criterion in dimension four}
Before giving the proof, let us prove some key lemmas for the proof.

\begin{lemma}\label{lem:finite-support-cohomology}
Let $X$ be a compact complex fourfold with no irreducible proper analytic subvarieties of complex codimension-one or $2$. Let $\cI\subseteq \cO_X$ be a nonzero coherent ideal sheaf. Then, for every line bundle $L$ on $X$, the natural map
\[
 H^q(X,L\otimes \cI)\longrightarrow H^q(X,L)
\]
is an isomorphism for every $q\geq 3$. In particular,
\[
H^3(X,L)\neq 0\quad\Longrightarrow\quad H^3(X,L\otimes \cI)\neq 0.
\]
\end{lemma}

\begin{proof}
If $\cI=\cO_X$, there is nothing to prove. Otherwise set $Q:=\cO_X/\cI$. This means
$
\dim \Supp(Q)\leq 1.
$
For a coherent sheaf supported in dimension at most $1$, one has
\[
H^q(X,L\otimes Q)=0\qquad(q\geq 2).
\]
The short exact sequence
\[
0\to L\otimes\cI\to L\to L\otimes Q\to 0
\]
therefore gives an isomorphism $H^q(X,L\otimes\cI)\simeq H^q(X,L)$ for every $q\geq 3$.
\end{proof}

\begin{lemma}\label{lem:dps}
Let $X$ be as in Lemma~\ref{thm:main}. Assume that $K$ is pseudo-effective. If $m>0$ and
$
H^3(X,K^{m+1})\neq 0,
$
then
$
H^0(X,\Omega_X^1\otimes K^m)\neq 0.
$
\end{lemma}

\begin{proof}
Since $K$ is pseudo-effective, $K^m$ is pseudo-effective. Choose a singular Hermitian metric $h_m$ on $K^m$ with semipositive curvature current, and let
$
\cI_m:=\mathcal I(h_m)
$
be its multiplier ideal sheaf. By Lemma~\ref{lem:finite-support-cohomology}, we have
\[
H^3(X,K^{m+1}\otimes \cI_m)\simeq H^3(X,K^{m+1})\neq 0.
\]
By \cite[Theorem~2.1]{DPS01}, there is a surjective map
\[
H^0(X,\Omega_X^1\otimes K^m\otimes \cI_m)
\twoheadrightarrow
H^3(X,K^{m+1}\otimes \cI_m).
\]
Thus
$
H^0(X,\Omega_X^1\otimes K^m\otimes \cI_m)\neq 0.
$
Since $\cI_m\subseteq \cO_X$, this gives a nonzero section of $\Omega_X^1\otimes K^m$.
\end{proof}
The following two lemmas are a consequence of standard results, but it is hard to find the reference in the form we need, so we provide a proof.
\begin{lemma}
\label{lem:ext}
Let $X$ be a smooth compact complex fourfold, let $Z\subset X$ be a closed subscheme
whose support has codimension at least $3$, and let $L$ be a line bundle on $X$.
Then the natural map
\[
H^1(X,L)\longrightarrow \Ext^1_X(I_Z,L)
\]
is an isomorphism. More generally, for line bundles $A$ and $B$ on $X$, there is
a natural isomorphism
\[
\Ext^1_X(I_Z\otimes A,B)\simeq H^1(X,B\otimes A^{-1}).
\]
\end{lemma}

\begin{proof}
If $Z=\varnothing$, then $I_Z=\cO_X$, and the assertion is immediate from
$\Ext^1_X(\cO_X,L)=H^1(X,L)$. Assume therefore that $Z\neq\varnothing$.
Consider the short exact sequence
\[
0\longrightarrow I_Z\longrightarrow \cO_X\longrightarrow \cO_Z\longrightarrow 0.
\]
Applying $\operatorname{RHom}_X(-,L)$ gives the exact segment
\[
\Ext^1_X(\cO_Z,L)
\longrightarrow
\Ext^1_X(\cO_X,L)
\longrightarrow
\Ext^1_X(I_Z,L)
\longrightarrow
\Ext^2_X(\cO_Z,L).
\]
Since $\Ext^1_X(\cO_X,L)=H^1(X,L)$, it is enough to prove that
$
\Ext^j_X(\cO_Z,L)=0
\text{ for } j=1,2.
$
That follows from \cite[Proposition 1.1.6]{HL} once we pass from local to global spectral sequence. The exact sequence above then gives an isomorphism
\[
H^1(X,L)=\Ext^1_X(\cO_X,L)
\xrightarrow{\sim}
\Ext^1_X(I_Z,L).
\]
Finally, since $A$ is locally free of rank one, there is a natural isomorphism
\[
\operatorname{RHom}_X(I_Z\otimes A,B)
\simeq
\operatorname{RHom}_X(I_Z,B\otimes A^{-1}).
\]
Taking first cohomology and applying the first part with $L=B\otimes A^{-1}$,
we obtain
\[
\Ext^1_X(I_Z\otimes A,B)
\simeq
\Ext^1_X(I_Z,B\otimes A^{-1})
\simeq
H^1(X,B\otimes A^{-1}).
\]
This proves the lemma.
\end{proof}

\begin{lemma}
\label{lem:rank-one-quotients}
Let $X$ be as in Lemma~\ref{thm:main}. Let $\mathscr A$ be a rank-two reflexive sheaf and let $L$ be a line bundle on $X$.
\begin{enumerate}[label=\textup{(\arabic*)}]
\item If $L\to \mathscr A$ is a nonzero morphism, then there exists a closed subscheme $Z\subset X$ whose support has codimension at least $3$ and an exact sequence
\[
0\to L\to \mathscr A\to I_Z\otimes \det(\mathscr A)\otimes L^{-1}\to 0.
\]
\item If $\mathscr A\to L$ is a nonzero morphism, then there exists a closed subscheme $Z\subset X$ whose support has codimension at least $3$ and an exact sequence
\[
0\to \det(\mathscr A)\otimes L^{-1}\to \mathscr A\to I_Z\otimes L\to 0.
\]
\end{enumerate}
\end{lemma}

\begin{proof}
We use the fact that in an exact sequence
\[
0\to \mathscr F\to \mathscr E\to \mathscr G\to 0,
\]
if $\mathscr E$ is reflexive and $\mathscr G$ is torsion-free, then
$\mathscr F$ is reflexive \cite[Lemma 15.23.5]{stacks-0AUY}. Moreover,
every rank-one reflexive sheaf on the smooth manifold $X$ is a line
bundle.

For \textup{(1)}, let $M\subset \mathscr A$ be the saturation of the
image of the nonzero morphism $L\to\mathscr A$. The morphism is
injective because $\mathscr A$ is torsion-free. Since
$\mathscr A/M$ is torsion-free, the cited lemma shows that $M$ is
reflexive; being of rank one, it is therefore a line bundle. The induced
nonzero morphism $L\to M$ corresponds to a nonzero section of
$M\otimes L^{-1}$. By Lemma~\ref{lem:line-sections}, this section is
nowhere vanishing, so $M\simeq L$. Hence
\[
0\to L\to\mathscr A\to Q\to0
\]
is exact with $Q$ torsion-free of rank one. Moreover,
\[
Q^{**}\simeq\det(\mathscr A)\otimes L^{-1},
\]
and therefore
\[
Q\simeq I_Z\otimes\det(\mathscr A)\otimes L^{-1}
\]
for some ideal sheaf $I_Z\subset\mathcal O_X$. Since $Q\to Q^{**}$
is an isomorphism in codimension one, $|Z|$ has codimension at least
two. The absence of codimension-two subvarieties then gives
$\operatorname{codim}_X|Z|\geq3$.

For \textup{(2)}, the image of a nonzero morphism
$\mathscr A\to L$ has rank one, and hence is $I_Z\otimes L$ for some
nonzero ideal sheaf $I_Z\subset\mathcal O_X$. Since $X$ has no
subvarieties of codimension one or two,
$\operatorname{codim}_X|Z|\geq3$. Thus
\[
0\to N\to\mathscr A\to I_Z\otimes L\to0
\]
is exact. The quotient is torsion-free, so the cited lemma implies that
$N$ is reflexive; hence $N$ is a line bundle. Finally,
$(I_Z)^{**}\simeq\mathcal O_X$, and taking determinants gives
\[
N\simeq\det(\mathscr A)\otimes L^{-1}.
\]
This proves both assertions.
\end{proof}
\begin{proposition}
\label{prop:foliation}
Let $X$ be as in Lemma~\ref{thm:main}, and suppose that $K=K_X$ is pseudo-effective. If, for some $m>0$,
\[
0\neq \omega\in H^0(X,\Omega_X^1\otimes K^{-m}),
\]
then $K$ is torsion.
\end{proposition}

\begin{proof}
The section $\omega$ is equivalently a nonzero morphism
\[
K^m\longrightarrow \Omega_X^1.
\]
Let $N^*\subset \Omega_X^1$ be the saturation of its sheaf-theoretic image.
Then $N^*$ is a rank-one reflexive sheaf. Since $X$ is smooth, $N^*$ is a
line bundle. The morphism $K^m\to\Omega_X^1$ factors through a nonzero morphism
\[
K^m\longrightarrow N^*.
\]
Equivalently, we get a nonzero section of the line bundle $N^*\otimes K^{-m}$.
By Lemma~\ref{lem:line-sections}, this section is nowhere vanishing. Hence
\[
N^*\simeq K^m.
\]
Since $K$ is pseudo-effective, $K^m$ is pseudo-effective. By Demailly's Frobenius integrability Theorem \ref{thm:demailly-frobenius} , applied with $L=K^m$ and
$\theta=\omega$, the kernel distribution
\[
S_\omega:=\{\xi\in T_X:\iota_\xi\omega=0\}
\]
is integrable. Equivalently, the saturated conormal line bundle $N^*\subset
\Omega_X^1$ defines a possibly singular codimension-one holomorphic foliation
$\mathcal G$. Its conormal bundle is
\[
N^*_{\mathcal G}=N^*\simeq K^m.
\]

\begin{claim}
The conormal bundle $N^*_{\mathcal G}$ is abundant.
\end{claim}

\begin{proof}[Proof of the claim]
Suppose not. Since $N^*_{\mathcal G}\simeq K^m$ is pseudo-effective, PRT
\cite[Theorem~4.2]{PRT} gives an integer $p\geq 2$ and an invertible subsheaf
\[
\mathcal M\subset \Omega_X^p
\]
such that
\[
\kappa(\mathcal M)=p>0.
\]
Therefore there exists $r>0$ such that
\[
h^0(X,\mathcal M^r)\geq 2.
\]
But $\mathcal M^r$ is a line bundle. By Lemma~\ref{lem:line-sections}, every
nonzero section of $\mathcal M^r$ is nowhere vanishing and trivializes
$\mathcal M^r$. Hence $\mathcal M^r\simeq\mathcal O_X$. Since $X$ is
connected and compact,
$
h^0(X,\mathcal O_X)=1,
$
a contradiction. Thus $N^*_{\mathcal G}$ is abundant.
\end{proof}
By abundance,
$
\kappa(N^*_{\mathcal G})=\nu(N^*_{\mathcal G}).
$
Using $N^*_{\mathcal G}\simeq K^m$, this gives
$
\kappa(K^m)=\nu(K^m).
$
Assume now that $K$ is not torsion. Then Lemma~\ref{lem:line-sections}
implies that
$
H^0(X,K^r)=0
\text{ for every }r>0,
$
because any nonzero section of $K^r$ would trivialize $K^r$. Hence
$
\kappa(K^m)=-\infty.
$
On the other hand, $K^m$ is pseudo-effective, so its numerical dimension satisfies
$
\nu(K^m)\geq 0.
$
This contradicts
$
\kappa(K^m)=\nu(K^m).
$
Therefore $K$ is torsion.
\end{proof}
We are now ready to prove Lemma \ref{thm:main}, idea is to analyze the kernel of the holomorphic two form as done in \cite{HoringPeternellRadloff2013}.
\begin{proof}[Proof of Lemma~\ref{thm:main}]
Assume, for contradiction, that $K$ is not torsion. Since $X$ is K\"ahler and has no codimension one subvarieties, this implies $H^{2,0}(X)\neq 0$, indeed, otherwise it would be projective (\cite[Chapter 7, Exercise 1]{Voisin02},\cite{KodairaEmbedding}). Choose
\[
0\neq \sigma\in H^0(X,\Omega_X^2)
\]
and choose linearly independent forms
\[
\eta_1,\eta_2\in H^0(X,\Omega_X^3).
\]
If $\sigma^2\neq 0$, then $\sigma^2$ is a nonzero section of $K$. By Lemma~\ref{lem:line-sections}, this gives that $K$ is a torsion, a contradiction. Hence,
\begin{equation}\label{eq:sigma-square-zero}
\sigma^2=0.
\end{equation}

Let $U\subset X$ be the complement of the zero locus of $\sigma$. Since this zero locus is a proper analytic subset and $X$ has no irreducible proper analytic subvarieties of complex codimension $1$ or $2$, the set $X\setminus U$ has codimension at least $3$. On $U$, the form $\sigma$ is nonzero and satisfies $\sigma^2=0$, so it has constant rank $2$. Define
\[
\cF_U:=\ker(\sigma:T_U\to \Omega_U^1),\qquad
\cE_U:=\Ann(\cF_U)\subset \Omega_U^1.
\]
Then $\cF_U$ and $\cE_U$ are rank-two vector bundles on $U$. Let
\[
\theta_\sigma:T_X\to \Omega_X^1
\]
be a contraction with $\sigma$. Define $\cE\subset \Omega_X^1$ to be the saturation of $\im(\theta_\sigma)$. Then $\cE$ is a rank-two reflexive sheaf and $\cE|_U=\cE_U$. Let
\[
\cF:=\ker\bigl(T_X\longrightarrow \cE^\vee\bigr),
\]
where the map is induced by inclusion $\cE\subset \Omega_X^1=T_X^\vee$. Since $\cE^\vee$ is torsion-free, $\cF$ is reflexive. Moreover $\cF$ has rank two and $\cF|_U=\cF_U$. Since $\cE_U=\Ann(\cF_U)$ and both $\cE$ and $\Ann(\cF)$ are reflexive subsheaves of $\Omega_X^1$ that agree on $U$, we have
\begin{equation}\label{eq:E-ann-F}
\cE=\Ann(\cF).
\end{equation}
On $U$, the decomposable two-form $\sigma$ is a nowhere-zero section of $\det \cE_U$. Since $X\setminus U$ has codimension at least two, this section extends to a nonzero section of $\det \cE$. By Lemma~\ref{lem:line-sections},
\begin{equation}\label{eq:detE}
\det\cE\simeq \cO_X.
\end{equation}
Moreover, on $U$ there is an exact sequence
\[
0\to \cE_U\to \Omega_U^1\to \cF_U^\vee\to 0.
\]
Taking determinants and extending reflexively across $X\setminus U$, we obtain
\[
K=\det \Omega_X^1\simeq \det\cE\otimes \det\cF^\vee.
\]
Using \eqref{eq:detE}, this gives
\begin{equation}\label{eq:detF}
\det\cF\simeq K^{-1}.
\end{equation}
The canonical isomorphism
\[
\Omega_X^3\simeq T_X\otimes K
\]
associates to $\eta_i$ a section
$
v_i\in H^0(X,T_X\otimes K).
$
Set
$
\beta_i:=\iota_{v_i}\sigma\in H^0(X,\Omega_X^1\otimes K).
$

\begin{claim}\label{claim:beta-annihilates}
Each $\beta_i$ annihilates $\cF$. Hence $\beta_i\in H^0(X,\cE\otimes K)$.
\end{claim}

\begin{proof}
It is enough to check this on $U$. Choose a local trivialization $\Omega$ of $K$ and write
\[
v_i=w_i\otimes \Omega
\]
for a local vector field $w_i$. For any local vector field $v$ on $U$,
\[
\beta_i(v)=\sigma(w_i,v)\Omega.
\]
If $v$ is a local section of $\cF_U$, then $\iota_v\sigma=0$, so $\sigma(v,u)=0$ for every local vector field $u$. Taking $u=w_i$ and using skew-symmetry gives
\[
\sigma(w_i,v)=-\sigma(v,w_i)=0.
\]
Thus $\beta_i$ annihilates $\cF_U$. Therefore $\beta_i$ lies in $H^0(U,\cE_U\otimes K)$. Since $\cE\otimes K$ is reflexive and $X\setminus U$ has codimension at least two,
\[
H^0(U,\cE_U\otimes K)=H^0(X,\cE\otimes K).
\]
\end{proof}
We now distinguish three cases.

\medskip
\noindent\textbf{Case 1: $\beta_1$ and $\beta_2$ are generically independent.}
Then their wedge is nonzero:
\[
0\neq \beta_1\wedge \beta_2\in H^0(X,\det\cE\otimes K^2).
\]
Using \eqref{eq:detE}, this gives a nonzero section of $K^2$, and Lemma~\ref{lem:line-sections} gives that $K$ is torsion, a contradiction.

\medskip
\noindent\textbf{Case 2: $\beta_1=\beta_2=0$.}
Then $v_1$ and $v_2$ lie in $H^0(X,\cF\otimes K)$. Therefore
\[
v_1\wedge v_2\in H^0(X,\det\cF\otimes K^2)
\simeq H^0(X,K)
\]
by \eqref{eq:detF}. This section is nonzero. Indeed, if $v_1\wedge v_2\equiv 0$, then $v_1$ and $v_2$ are generically proportional. The proportionality ratio is a meromorphic function on $X$, hence constant by Lemma~\ref{lem:no-meromorphic}; this would make $\eta_1$ and $\eta_2$ linearly dependent. Thus $H^0(X,K)\neq 0$, and Lemma~\ref{lem:line-sections} again gives that $K$ is torsion, a contradiction.

\medskip
\noindent\textbf{Case 3: $\beta_1$ and $\beta_2$ span a rank-one subsystem.}
Since the first two cases lead to contradictions, this is the only possible case under the standing assumption that $K$ is not torsion.

\begin{claim}\label{claim:basis-change}
After replacing $\eta_1,\eta_2$ by a suitable basis of their span, we may assume
\[
\beta_1\neq 0,
\qquad
\beta_2=0.
\]
\end{claim}

\begin{proof}
The assignment
\[
\eta_i\longmapsto v_i\longmapsto \beta_i=\iota_{v_i}\sigma
\]
is linear. Since $\beta_1,\beta_2$ are not both zero, assume after interchanging indices that $\beta_1\neq 0$. If $\beta_2=0$, there is nothing to prove. Otherwise $\beta_1$ and $\beta_2$ are nonzero and generically proportional. Their local proportionality quotients glue to a meromorphic function on $X$. By Lemma~\ref{lem:no-meromorphic}, this meromorphic function is constant. Thus $\beta_2=c\beta_1$ on a dense open set for some $c\in\mathbb C$, and hence everywhere because $\cE\otimes K$ is torsion-free. Replacing $\eta_2$ by $\eta_2-c\eta_1$ gives the claim.
\end{proof}
The section
\[
0\neq \beta_1\in H^0(X,\cE\otimes K)
\]
is a nonzero morphism
\[
K^{-1}\to \cE.
\]
Using Lemma~\ref{lem:rank-one-quotients} and \eqref{eq:detE}, we obtain an exact sequence
\begin{equation}\label{eq:initial-extension}
0\to K^{-1}\to \cE\to I_Z\otimes K\to 0
\end{equation}
for some closed subscheme $Z\subset X$ whose support has codimension at least $3$. Its extension class lies in
\begin{equation}\label{eq:initial-ext-group}
\Ext^1_X(I_Z\otimes K,K^{-1})\simeq \Ext^1_X(I_Z,K^{-2})\simeq H^1(X,K^{-2}),
\end{equation}
where the last isomorphism is Lemma~\ref{lem:ext}. If the class in \eqref{eq:initial-ext-group} were zero, then \eqref{eq:initial-extension} would split over $X\setminus |Z|$. Hence over $X\setminus |Z|$ there would be a nonzero inclusion
\[
K\hookrightarrow \cE\subset \Omega_X^1,
\]
or equivalently a nonzero section
$
0\neq \omega\in H^0(X\setminus |Z|,\Omega_X^1\otimes K^{-1}).
$
By Hartogs extension,
$
0\neq \omega\in H^0(X,\Omega_X^1\otimes K^{-1}).
$
Proposition~\ref{prop:foliation} gives that $K$ is torsion, a contradiction. Therefore the class in \eqref{eq:initial-ext-group} is nonzero. We have constructed
\begin{equation}\label{eq:xi2}
0\neq \xi_2\in H^1(X,K^{-2}).
\end{equation}
Moreover, from $\beta_2=0$ we retain the fixed nonzero section
\begin{equation}\label{eq:fixed-v2}
0\neq v_2\in H^0(X,\cF\otimes K).
\end{equation}
Let $\mathcal M\subset \mathbb Z_{>0}$ be the smallest subset satisfying
\[
2\in \mathcal M,
\qquad
m\in\mathcal M\Longrightarrow 2m\in\mathcal M\text{ and }2m+1\in\mathcal M.
\]
We now begin the recursion. Suppose that, for some $m\in\mathcal M$, we have constructed
\begin{equation}\label{eq:xi-m}
0\neq \xi_m\in H^1(X,K^{-m}).
\end{equation}
By Serre duality on the smooth fourfold $X$,
$
H^1(X,K^{-m})^\vee\simeq H^3(X,K^{m+1}).
$
Thus
\begin{equation}\label{eq:h3-nonzero}
H^3(X,K^{m+1})\neq 0.
\end{equation}
By Lemma~\ref{lem:dps}, choose
\begin{equation}\label{eq:alpha-m}
0\neq \alpha_m\in H^0(X,\Omega_X^1\otimes K^m).
\end{equation}
We view $\alpha_m$ as a morphism
\[
\alpha_m:T_X\to K^m.
\]

\medskip
\noindent\textbf{Recursive Case A: $\cF\subseteq \ker(\alpha_m)$.}
Then, by \eqref{eq:E-ann-F}, $\alpha_m$ is a nonzero section of $\cE\otimes K^m$, hence gives a nonzero morphism
\[
K^{-m}\to \cE.
\]
Using Lemma~\ref{lem:rank-one-quotients} and \eqref{eq:detE}, we obtain an exact sequence
\begin{equation}\label{eq:caseA-extension}
0\to K^{-m}\to \cE\to I_Z\otimes K^m\to 0
\end{equation}
for some closed subscheme $Z\subset X$ whose support has codimension at least $3$. Its extension class lies in
\begin{equation}\label{eq:caseA-ext-group}
\Ext^1_X(I_Z\otimes K^m,K^{-m})\simeq \Ext^1_X(I_Z,K^{-2m})\simeq H^1(X,K^{-2m}).
\end{equation}
If this class were zero, then \eqref{eq:caseA-extension} would split over $X\setminus |Z|$, producing a nonzero section
\[
0\neq \omega_m\in H^0(X\setminus |Z|,\Omega_X^1\otimes K^{-m}).
\]
Hartogs extension gives
\[
0\neq \omega_m\in H^0(X,\Omega_X^1\otimes K^{-m}),
\]
and Proposition~\ref{prop:foliation} gives that $K$ is torsion, a contradiction. Hence the class in \eqref{eq:caseA-ext-group} is nonzero. Define
\begin{equation}\label{eq:xi-2m}
0\neq \xi_{2m}\in H^1(X,K^{-2m}).
\end{equation}
Since $2m\in\mathcal M$, the recursion continues.

\medskip
\noindent\textbf{Recursive Case B: $\cF\not\subseteq \ker(\alpha_m)$.}
Then the restriction
\[
\alpha_m|_{\cF}:\cF\to K^m
\]
is nonzero and generically surjective. Using Lemma~\ref{lem:rank-one-quotients} and \eqref{eq:detF}, we obtain an exact sequence
\begin{equation}\label{eq:caseB-extension}
0\to K^{-m-1}\to \cF\to I_Z\otimes K^m\to 0
\end{equation}
for some closed subscheme $Z\subset X$ whose support has codimension at least $3$. Its extension class lies in
\begin{equation}\label{eq:caseB-ext-group}
\Ext^1_X(I_Z\otimes K^m,K^{-m-1})\simeq \Ext^1_X(I_Z,K^{-2m-1})\simeq H^1(X,K^{-2m-1}).
\end{equation}
If this class is nonzero, define
\begin{equation}\label{eq:xi-2m-plus1}
0\neq \xi_{2m+1}\in H^1(X,K^{-2m-1}).
\end{equation}
Since $2m+1\in\mathcal M$, the recursion continues. It remains to exclude the possibility that the class in \eqref{eq:caseB-ext-group} is zero. If it were zero, then \eqref{eq:caseB-extension} would split over $X\setminus |Z|$:
\[
\cF|_{X\setminus |Z|}\simeq K^{-m-1}|_{X\setminus |Z|}\oplus K^m|_{X\setminus |Z|}.
\]
After tensoring by $K$,
\[
(\cF\otimes K)|_{X\setminus |Z|}\simeq K^{-m}|_{X\setminus |Z|}\oplus K^{m+1}|_{X\setminus |Z|}.
\]
Restrict the fixed nonzero section \eqref{eq:fixed-v2}:
\[
0\neq v_2|_{X\setminus |Z|}\in H^0(X\setminus |Z|,\cF\otimes K).
\]
Under the splitting, write
$
v_2|_{X\setminus |Z|}=(s_-,s_+),
$
where
\[
s_-\in H^0(X\setminus |Z|,K^{-m}),
\qquad
s_+\in H^0(X\setminus |Z|,K^{m+1}).
\]
At least one of $s_-$ and $s_+$ is nonzero. By Hartogs extension, we obtain a nonzero global section of either $K^{-m}$ or $K^{m+1}$. Lemma~\ref{lem:line-sections} gives that $K$ is torsion, a contradiction. Therefore the class in \eqref{eq:caseB-ext-group} is nonzero, and \eqref{eq:xi-2m-plus1} holds. Thus, starting from \eqref{eq:xi2}, the recursion never terminates. We obtain a sequence
\[
0\neq \xi_{m_j}\in H^1(X,K^{-m_j}) \quad \text{ with } \quad
m_0=2,
\quad
m_{j+1}\in\{2m_j,2m_j+1\}.
\]
In particular $m_j\to \infty$. By Serre duality,
$
H^3(X,K^{m_j+1})\neq 0
$
for every $j$. This is a contradiction to the assumption that $X$ has no codimension one subvarieties, due to Lemma \ref{lem:dps},\ref{lem:line-sections} and \cite[Proposition 2.6]{AH22}. The contradiction shows that $K$ cannot be non-torsion. Hence $K_X$ is torsion.
\end{proof}
\subsection{ Irregularity.}
We first prove the following standard lemmas.
\begin{lemma}
\label{lem:finite-etale-trivialization-higher-genus}
Let $f: Y \to B$ be a smooth proper holomorphic submersion between connected complex manifolds. Assume that every fiber is biholomorphic to the same curve $C$ of genus $g\geq 2$. Then there is finite \'etale cover $B'\to B$, such that the pulled back family is trivial.
\end{lemma}

\begin{proof}
By the Fischer-Grauert theorem, $f$ is locally analytically trivial. Since the group $\operatorname{Aut}(C)$ is finite the transition functions are locally constant. Therefore the fibration is defined by a finite monodromy representation
$
\rho:\pi_1(B)\longrightarrow \operatorname{Aut}(C).
$
Now choose the finite covering  $B'\to B$ corresponding to the
subgroup $\ker(\rho)\subset \pi_1(B)$. After pulling back the family to $B'$, the monodromy representation is trivial and hence globally trivial $Y\times_B B' \simeq C\times B'$.
\end{proof}
\begin{lemma}\label{lem:elliptic-isotrivial-principal-after-etale}
Let $f:Y\to B$ be a smooth proper holomorphic submersion between connected compact complex manifolds. Assume that every fiber is biholomorphic to the same elliptic curve $E$. Then there exists a finite \'etale cover $B' \to B$ such that the pulled-back fibration is a principal $E$-bundle.
\end{lemma}

\begin{proof}
Due to Fischer-Grauert theorem, $f$ is locally analytically trivial. After choosing an origin on $E$, we have an exact sequence
\[
    0\longrightarrow E\longrightarrow \operatorname{Aut}(E)
    \longrightarrow \operatorname{Aut}(E,0)\longrightarrow 0,
\]
Composing the transition functions with
$
\operatorname{Aut}(E)\longrightarrow \operatorname{Aut}(E,0),
$
we obtain transition functions with values in the finite group
$\operatorname{Aut}(E,0)$. Since $\operatorname{Aut}(E,0)$ is finite transition functions are locally constant. Equivalently, they define a finite monodromy representation
\[
    \rho:\pi_1(B)\longrightarrow \operatorname{Aut}(E,0).
\]
Observe that the subgroup $\ker(\rho)\subset \pi_1(B)$ has finite index. Let $B'\longrightarrow B$ be the finite \'etale covering corresponding to $\ker(\rho)$. After pulling back to $B'$, the induced monodromy representation to $\operatorname{Aut}(E,0)$ is trivial. Therefore,  all transition functions of the pulled-back fibration take values in the translation subgroup $E\subset \operatorname{Aut}(E)$. Thus $Y'\to B'$ is a principal $E$-bundle.
\end{proof}

\begin{lemma}\label{lem:q-geq-four-torus}
Let $Y$ be a compact K\"ahler fourfold with $q(Y)\geq 4$. Assume that $Y$ contains no divisors and no surfaces. Then $Y$ is a complex torus and therefore $q(Y)=4$.
\end{lemma}

\begin{proof}
Since $Y$ contains no divisors, its algebraic dimension is zero. Hence, by the standard surjectivity theorem for the Albanese map of a compact K\"ahler manifold of algebraic dimension zero, the Albanese map
\[
\alpha_Y:Y\longrightarrow \operatorname{Alb}(Y)
\]
is surjective, \cite[Chapter 13, Lemma 13.1]{Ueno1975}. Since
$
\dim Y=4
$
and
\[
\dim \operatorname{Alb}(Y)=q(Y)\geq 4,
\]
surjectivity forces $q(Y)=4$. Thus $\alpha_Y$ is a surjective generically finite holomorphic map from the smooth fourfold $Y$ onto the four-dimensional complex torus $\operatorname{Alb}(Y)$. We show that $\alpha_Y$ is finite \'etale. Consider the differential
\[
d\alpha_Y:T_Y\longrightarrow b^*T_{\operatorname{Alb}(Y)}.
\]
Both vector bundles have rank $4$. Since $\alpha_Y$ is generically finite and $Y$ contains no divisors, $\det(d\alpha_Y)$ is nowhere vanishing, so $\alpha_Y$ is finite \'etale by compactness, and hence $Y$ is a complex torus.
\end{proof}
\begin{lemma}\label{lem:q-three-impossible}
    Let $X$ be a compact K\"ahler fourfold. Assume that $K_X$ is pseudo-effective and that $X$ contains no analytic subvarieties of codimension $1$ or $2$. Then 
    $
    q(X)\neq 3. 
    $
\end{lemma}
\begin{proof}
    Assume for a contradiction that $q(X)=3$. Let, 
    \[ 
    \alpha_X:X\longrightarrow A:=\operatorname{Alb}(X) 
    \] 
    be the Albanese map. Then $A$ is a three-dimensional complex torus.
    \begin{claim}\label{claim:smooth}
        The Albanese map is smooth with curve fibers.
    \end{claim}
    \begin{proof}
    $\alpha_X$ is surjective and its general fiber is a curve; \cite[Chapter 13, Lemma 13.1]{Ueno1975}. The critical locus of $\alpha_X$ is the rank drop locus of the following differential,
    \[
    d\alpha_X: T_X \longrightarrow \alpha_X^*T_A.
    \]
    The rank drop locus of $d\alpha_X$ has codimension at most $2$ \cite[Chapter 14, Section 4]{FultonIntersectionTheory}. Therefore, if it is nonempty, we have a contradiction to the assumptions of the theorem.
    \end{proof}
    Now we prove that the fibers of the map are not rational curves and then argue using case analysis of the genus of the curves.
    \begin{claim}\label{claim:notrational}
        The fibers of $\alpha_X$ are not rational curves.
    \end{claim}
    \begin{proof}
        Due to Claim \ref{claim:smooth} fibers are smooth. Let $F$ be a fiber of $\alpha_X$. Since $A$ is a torus, $K_A\simeq \mathcal O_A$. Thus, by adjunction,
        \[
        K_X|_F \simeq K_F.
        \]
        The fibers of $\alpha_X$ form a covering family of $X$, since $K_X$ is pseudo-effective it has non negative degree with movable curve class; see \cite[Proposition 1.4(iii)]{BDPP}. Hence
        \[
        \operatorname{deg}(K_F)= K_X \cdot F \geq 0.
        \]
        Therefore the genus of the fiber $g(F) \geq 1$.
        \end{proof}
        \begin{claim}
            $\alpha_X$ is isotrivial.
        \end{claim}
        \begin{proof}
             We give the standard argument for completeness. Pull the family back to the universal cover
             \[
             \mathbb C^3\longrightarrow A.
             \]
             Since $\mathbb C^3$ is simply connected, the pulled-back family admits a holomorphic classifying map to Teichm\"uller space
             \[
             \mathbb C^3\longrightarrow \mathcal T_g,
             \]
             where $g=g(F)$. Due to claim \ref{claim:notrational}, we have $g(F)\geq 1$. If $g=1$, then $\mathcal T_1\simeq \mathbb H$, which is biholomorphic to the unit disc. If $g\geq 2$, then by the Bers embedding theorem, $\mathcal T_g$ is realised as a bounded domain in a complex affine space; see \cite{Bers1960}. In either case, the classifying map is a bounded holomorphic map from $\mathbb C^3$ to a complex affine space. By Liouville's theorem, it is constant. Hence $\alpha_X$ is isotrivial.
        \end{proof}
        Now we split the proof into $2$ cases.
        \begin{description}
            \item[Case 1 (genus is strictly greater than one)]
            By Lemma \ref{lem:finite-etale-trivialization-higher-genus}, after a finite proper surjective base change
            \[
            A'\longrightarrow A
            \]
            the pulled-back family becomes a product
            \[
            X':=X\times_A A'\simeq C\times A',
            \]
            where $C$ is a fixed smooth curve of genus $g(F)\geq 2$. Let
            \[
            \pi:X'\longrightarrow X
            \]
            be the natural projection. Then
            $
            D':=\{p\}\times A'
            $
            is a codimension-one analytic subset of $X'$. Since $\pi$ is finite, it cannot contract $D'$. Hence, by Remmert's proper mapping theorem,
            $
            \pi(D')\subset X
            $
            is an analytic subset of dimension $3$. This contradicts the assumption that $X$ contains no analytic subvarieties of codimension $1$. Therefore the case $g(F)\geq 2$ is impossible.
            \item[Case 2 (genus is equal to one)] Then $\alpha_X:X\to A$ is a smooth isotrivial fibration by elliptic curves. By Fischer-Grauert \cite{FischerGrauert}, it is locally analytically trivial. After a finite \'etale base change killing the finite automorphism part of the monodromy, the pulled-back fibration
            \[
            \alpha_X':X':=X\times_A A'\longrightarrow A'
            \]
            is a principal bundle under a fixed elliptic curve $E$; due to Lemma \ref{lem:elliptic-isotrivial-principal-after-etale}. The manifold $X'$ is again compact K\"ahler. Since $X'$ is K\"ahler, Blanchard's theorem, in the form stated in \cite[Theorem 1.7]{Hofer1993} and \cite[Proposition 5.2]{Hofer1993}, implies that
            \[
            H^1(X',\mathbb C)\simeq H^1(A',\mathbb C)\oplus H^1(E,\mathbb C).
            \]
            Therefore,
            \[
            b_1(X')=b_1(A')+b_1(E)=6+2=8.
            \]
            Since $X'$ is compact K\"ahler, we get
            \[
            q(X')=\frac{1}{2}b_1(X')=4.
            \]
            $X'$ is a complex torus due to Lemma \ref{lem:q-geq-four-torus}. Replacing $X'\to X$ by a finite Galois \'etale cover, we obtain a finite Galois \'etale cover
            \[
            T\longrightarrow X
            \]
            where $T$ is a complex torus. Let $G$ be the Galois group, so that $X=T/G$. Then
            \[
            H^0(X,\Omega_X^1)=H^0(T,\Omega_T^1)^G.
            \]
            Since $q(X)=3$ and $\dim T=4$, the invariant subspace $H^0(T,\Omega_T^1)^G$ is codimension one in $H^0(T,\Omega_T^1)$. Hence, some $g\in G$ acts nontrivially on $H^0(T,\Omega_T^1)$ while fixing $H^0(T,\Omega_T^1)^G$. Write $T=V/\Lambda$, where $V\simeq \mathbb C^4$ and $\Lambda$ is a lattice. Let $L_g:V\to V$ be the linear part of $g$. The fixed hyperplane on $H^0(T,\Omega_T^1)\simeq V^*$ is dual to the fixed subspace
            \[
            V_g:=\ker(L_g-\operatorname{id}_V)\subset V.
            \]
            This subspace has dimension $3$. Since $L_g$ preserves the lattice $\Lambda$, $\Lambda_g= \Lambda \cap V_{g}$ is a lattice in $V_g$. Therefore, $D=V_{g}/\Lambda_g$ is a threedimensional subtorus of $T$, therefore gives a divisor in $X$, this contradiction removes the elliptic case as well, therefore $q(X)\neq 3$.
        \end{description}
\end{proof}

\begin{proof}[Proof of Lemma \ref{thm:0or4}:]
    Follows from, Lemma \ref{lem:q-geq-four-torus}, Lemma \ref{lem:q-three-impossible} and \cite[Chapter 13, Lemma 13.1]{Ueno1975}.
\end{proof}
Let us note some corollaries.
\begin{corollary}\label{cor:euler0}
    Let $X$ be a compact K\"ahler fourfold, with $K_X$ pseudo-effective. Assume that 
    \begin{itemize}
        \item $\chi(X,\mathcal{O}_X)\leq0$;
        \item $X$ has no codimension one subvarieties ;
        \item $X$ has no codimension two subvarieties.
    \end{itemize}
    Then $X$ is either a quotient of a complex torus or an irreducible holomorphic symplectic manifold.
\end{corollary}

\begin{proof}
 By Lemma \ref{thm:0or4}, $q(X)=0$ or $q(X)=4$. If $q(X)=4$, then Lemma \ref{lem:q-geq-four-torus} shows that $X$ is a complex torus, and we are done. Assume $q(X)=0$. We first prove that $K_X$ is torsion. Suppose not. Then
Lemma \ref{lem:line-sections} gives
\[
H^0(X,K_X)=0,
\]
so $h^{0,4}(X)=0$. Since $X$ has no divisors, so
$
h^{0,2}(X)>0;
$
due to \cite[Chapter 7, Exercise 1]{Voisin02},\cite{KodairaEmbedding}. Using $q(X)=h^{0,1}(X)=0$ and $\chi(X,\mathcal O_X)\leq0$, we get
\[
0\geq\chi(X,\mathcal O_X)
  =1+h^{0,2}(X)-h^{0,3}(X).
\]
Hence
\[
h^{0,3}(X)\geq1+h^{0,2}(X)\ge2.
\]
Lemma \ref{thm:main} applies and gives that $K_X$ is torsion, contradiction. Therefore
$K_X$ is torsion. Finally, \cite[Theorem 2.8]{Vikash2026ClassificationOP} gives the required conclusion.
\end{proof}
\subsection{The reduction to nef case}\label{sec:nef}
\begin{proposition}
\label{prop:propleq1}
Let $X$ be a compact K\"ahler manifold of dimension $n$, and let $L$ be a pseudo-effective line bundle on $X$ such that
$
c_1(L)_{\mathbb R}\neq 0.
$
Let $h$ be a singular Hermitian metric on $L$ with semipositive curvature current, and set
$
\mathcal I_m:=\mathcal I(h^m).
$
Assume that
\begin{enumerate}
    \item $V(\mathcal I_m)$ has dimension at most $1$ for every $m\geq 1$;
    \item $\chi(X,K_X)>0$ and $\chi(X,K_X\otimes L^m)\geq 0$ for every $m\geq 1$.
\end{enumerate}
Then $X$ admits a codimension-one analytic subset.
\end{proposition}

\begin{proof}
Assume, for contradiction, that $X$ contains no codimension-one analytic subsets. Since $c_1(L)_{\mathbb R}\neq 0$, the line bundle $L$ is not torsion. For large $m$ we have $H^0(X,K_X\otimes L^m)=0$ and $H^0(X,L^{-m})=0$ because, if $H^0(X,K_X\otimes L^m)\neq 0$ then either we have codimension one subvariety or we have $c_1(K_X)+mc_1(L)=0$ and the later cannot be true for large $m$. If we have $H^0(X,L^{-m})\neq0$ for large $m$ then either we have a codimension one subvariety or $c_1(L)_{\mathbb{R}}=0$, both are not possible due to our assumptions. By Hirzebruch-Riemann-Roch, $\chi(X,K_X\otimes L^m)$ is a polynomial in $m$. This polynomial cannot be zero for infinitely many $m$ because then the polynomial will be identically $0$ which is not possible because of our assumption $\chi(X,K_X)>0$. Since $\chi(X,K_X\otimes L^m)\geq 0$ for every $m$, we have
$
\chi(X,K_X\otimes L^m)>0
$
for large $m$. For such $m$, the vanishing of $H^0(X,K_X\otimes L^m)$ and $H^n(X,K_X\otimes L^m) \simeq H^0(X,L^{-m})^{\vee}$ gives
\[
\chi(X,K_X\otimes L^m)
=
\sum_{q=1}^{n-1}(-1)^q h^q(X,K_X\otimes L^m).
\]
Due to assumptions $\chi(X,K_X\otimes L^m)$ is positive, so there is a even integer $2k$ with $1\leq 2k\leq n-1$, such that $H^{2k}(X,K_X\otimes L^m)\neq 0$. Taking sub sequence we may assume that the same integer $2k$ works for unbounded sequence $m_i$. For each $i$ consider the following exact sequence 
\[
0\longrightarrow \mathcal I_{m_i}\otimes K_X\otimes L^{m_i}
\longrightarrow K_X\otimes L^{m_i}
\longrightarrow \mathcal Q_{m_i}
\longrightarrow 0.
\]
Due to assumption $1$, $\mathcal Q_{m_i}$ is supported in dimension at most $1$. Hence
\[
H^q(X,\mathcal Q_{m_i})=0
\qquad\text{for all }q\geq 2.
\]
Therefore the natural map
\[
H^{2k}(X,\mathcal I_{m_i}\otimes K_X\otimes L^{m_i})
\longrightarrow
H^{2k}(X,K_X\otimes L^{m_i})
\]
is surjective. Since the target is nonzero, we get
\[
H^{2k}(X,\mathcal I_{m_i}\otimes K_X\otimes L^{m_i})\neq 0.
\]
Now apply the pseudo-effective Hard Lefschetz theorem with multiplier ideals, Theorem~\ref{thm:DPS-HL}, to the pseudo-effective line bundle $(L^{m_i},h^{m_i})$. We obtain
\[
H^0\!\left(X,\Omega_X^{n-2k}\otimes L^{m_i}\otimes \mathcal I_{m_i}\right)\neq 0.
\]
Since $\mathcal I_{m_i}\hookrightarrow \mathcal O_X$, this gives
$
H^0\!\left(X,\Omega_X^{n-2k}\otimes L^{m_i}\right)\neq 0
$
for infinitely many $m_i$. By \cite[Proposition~2.6]{AH22}, we have a contradiction to $c_1(L)_{\mathbb R}\neq 0$. Therefore, $X$ must contain a codimension-one analytic subset.
\end{proof}
\begin{lemma}\label{thm:nef-positive-euler-torsion}
Let $X$ be a compact K\"ahler fourfold. Assume that $K_X$ is nef and that
$X$ contains no irreducible analytic subvarieties of codimension $1$ or
$2$. If  $ \chi(X,\mathcal O_X)>0, $ then $K_X$ is torsion.
\end{lemma}

\begin{proof}
Assume, for contradiction, that $K_X$ is not torsion. By Lemma~\ref{thm:0or4}, we have
\[
    q(X)=0
    \quad\text{or}\quad
    q(X)=4.
\]
If $q(X)=4$, then Lemma~\ref{lem:q-geq-four-torus} gives that $X$ is a
complex torus. Hence
$
    K_X\simeq \mathcal O_X,
$
contradicting the assumption that $K_X$ is not torsion. Therefore
$
    q(X)=0.
$
We claim that
$
    c_1(K_X)_{\mathbb R}\neq 0.
$
Indeed, if $c_1(K_X)_{\mathbb R}=0$, then $c_1(K_X)\in H^2(X,\mathbb Z)$
is torsion. Thus, for some $r>0$,
$
    c_1(K_X^r)=0.
$
Since $q(X)=0$, we have $\operatorname{Pic}^0(X)=0$. Hence every
topologically trivial line bundle is trivial, and therefore
$
    K_X^r\simeq \mathcal O_X.
$
This contradicts the assumption that $K_X$ is not torsion. Hence
\[
    c_1(K_X)_{\mathbb R}\neq 0.
\]
We now apply Proposition~\ref{prop:propleq1} with
$
    L:=K_X.
$
Since $K_X$ is nef, choose a singular Hermitian metric $h$ on $K_X$ with semipositive curvature current. For each $m\geq 1$, the nonzero multiplier ideal $\mathcal I_m:=\mathcal I(h^m)$ has proper cosupport, which has dimension at most one since $X$ has no divisors or surfaces. It remains to check the Euler characteristic hypotheses of Proposition~\ref{prop:propleq1}. Since $\dim X=4$, Serre duality gives
\[
    \chi(X,K_X)=\chi(X,\mathcal O_X)>0.
\]
Next, for $s\geq 1$, Hirzebruch-Riemann-Roch gives
\[
    \chi(X,K_X^s)
    =
    \chi(X,\mathcal O_X)
    +
    \frac{s(s-1)}{24}\,c_1(K_X)^2\cdot c_2(X)
    +
    \frac{s^2(s-1)^2}{24}\,c_1(K_X)^4.
\]
Since $K_X$ is nef,
$
    c_1(K_X)^4\geq 0.
$
Moreover, by the Miyaoka-Yau inequality for compact K\"ahler manifolds with
nef canonical bundle \cite[Theorem~1.1]{LiuMY}, we have
\[
    \bigl(2(n+1)c_2(X)-n c_1(X)^2\bigr)\cdot c_1(K_X)^{n-2}\geq 0.
\]
Taking $n=4$, and using $c_1(X)=-c_1(K_X)$, this becomes
\[
    \bigl(10c_2(X)-4c_1(K_X)^2\bigr)\cdot c_1(K_X)^2\geq 0.
\]
Therefore
\[
    c_1(K_X)^2\cdot c_2(X)
    \geq
    \frac{2}{5}c_1(K_X)^4
    \geq 0.
\]
Hence, for every $s\geq 1$,
\[
    \chi(X,K_X^s)\geq \chi(X,\mathcal O_X)>0.
\]
In particular, for every $m\geq 1$,
\[
    \chi(X,K_X\otimes L^m)
    =
    \chi(X,K_X^{m+1})
    \geq 0.
\]
Proposition~\ref{prop:propleq1} then implies that $X$ contains a
codimension-one analytic subset, a contradiction. Therefore $K_X$ must be torsion.
\end{proof}
\begin{proof}[Proof of Lemma \ref{thm:mainnef}]
    We divide the proof in cases depending on $\chi(X,\mathcal{O}_X)$.
    \begin{description}
        \item[Case 1 ($\chi(X,\mathcal{O}_X)\leq 0$)] This case follows from Corollary \ref{cor:euler0}
        \item[Case 2 ($\chi(X,\mathcal{O}_X)> 0$)] This case follows from Lemma \ref{thm:nef-positive-euler-torsion}.
    \end{description}
\end{proof}
\begin{proof}[Proof of Lemma \ref{lem:pseff-canonical-nef-no-divisors-surfaces}:]
Suppose, for contradiction, that $K_X$ is not nef. Since $K_X$ is
pseudo-effective and $\dim X=4$, Cao-Höring
\cite[Corollary~1.4]{CaoHoering2020} gives a rational curve $C$ such that
\[
    f:\mathbb P^1\longrightarrow X
\]
such that
$
    c_1(K_X)\cdot f(\mathbb P^1)<0.
$
After replacing $f$ by the normalization of its image, we may assume that
$f$ is generically injective onto its image. Put
$
    C:=f(\mathbb P^1).
$
Then
$
    K_X\cdot C<0.
$
Let $H$ be an irreducible component of
$\operatorname{Hol}(\mathbb P^1,X)$ containing $f$. By the standard
deformation estimate for holomorphic maps
\cite{Horikawa1973},
\[
    \dim_f H
    \geq
    \chi(\mathbb P^1,f^*T_X).
\]
By Riemann-Roch on $\mathbb P^1$,
\[
    \chi(\mathbb P^1,f^*T_X)
    =
    \operatorname{rk}(T_X)+\deg(f^*T_X)
    =
    4-K_X\cdot C.
\]
Since $K_X\cdot C<0$, it follows that
\[
    \dim_f H\geq 5.
\]
The group $\operatorname{Aut}(\mathbb P^1)$ has dimension $3$. Since
$f$ is generically injective onto its image, the maps near $f$ with the
same image cycle as $f$ are obtained, up to a finite ambiguity, by
reparametrization. Hence the corresponding germ of image cycles has dimension
at least
\[
    \dim_f H-3\geq 2.
\]
In particular, after replacing $H$ by a suitable analytic curve germ through
$f$, we obtain a holomorphic family
\[
    f_t:\mathbb P^1\longrightarrow X,\qquad t\in(\Delta,0),
\]
with $f_0=f$, such that the effective one-cycles
\[
    \Gamma_t:=(f_t)_*[\mathbb P^1]
\]
are not all equal. Let $\omega$ be a Kähler form on $X$. Since $K_X\cdot C<0$, choose
$0<\varepsilon\ll 1$ such that
\[
    \bigl(c_1(K_X)+\varepsilon[\omega]\bigr)\cdot C<0.
\]
Set
\[
    \alpha:=c_1(K_X)+\varepsilon[\omega].
\]
Since $c_1(K_X)$ is pseudo-effective, the class $\alpha$ is big: indeed,
if $R\in c_1(K_X)$ is a closed positive current, then $R+\varepsilon\omega$
is a Kähler current in the class $\alpha$. By Demailly's regularization
theorem for Kähler currents, we may choose such a current $T\in\alpha$ with
analytic singularities
\cite[Theorem~1.1]{Demailly1992}; see also
\cite[Section~3]{DemaillyPaun2004}. Let
$
    S:=\operatorname{Sing}(T).
$
Then $S\subsetneq X$ is an analytic subset. Since $X$ contains no divisors or surfaces, $\dim S\leq 1$. After shrinking
$\Delta$, the maps $f_t$ form a continuous family, so the homology class
$(f_t)_*[\mathbb P^1]$ is locally constant. Hence
\[
\alpha\cdot\Gamma_t
=
\alpha\cdot\Gamma_0
=
\alpha\cdot C
<0
\]
for every $t\in\Delta$. We claim that $\operatorname{Supp}(\Gamma_t)\subset S$ for every
$t\in\Delta$. Suppose otherwise. Since $T$ has analytic singularities and
$f_t(\mathbb P^1)$ is not contained in $S$, the pullback $f_t^*T$ is a
well-defined positive current on $\mathbb P^1$. As $T$ represents $\alpha$,
we obtain
\[
0\leq\int_{\mathbb P^1}f_t^*T
=\alpha\cdot\Gamma_t<0,
\]
a contradiction. Therefore every cycle $\Gamma_t$ is inside $S$. But $S$ is subvariety of dimension at most $1$. Hence $S$ has only finitely many irreducible curve components, so any connected analytic family of effective one-cycles supported on $S$ is constant, this contradicts the choice of the family $\{\Gamma_t\}$, therefore $K_X$ is nef.
\end{proof}

\begin{proof}[Proof of Theorem \ref{thm:thmmain}]
    Follows from Lemma \ref{lem:pseff-canonical-nef-no-divisors-surfaces} and Lemma \ref{thm:mainnef}.
\end{proof}
Using this we can prove a version of Conjecture \ref{conj:simple-kahler}.
\begin{corollary}
    Let $X$ be a Simple compact K\"ahler fourfold, such that the only subvarieties are curves. If $K_X$ is pseudo-effective then $K_X$ is torsion.
\end{corollary}
\begin{proof}
    Immediate consequence of Theorem \ref{thm:thmmain}.
\end{proof}

\begin{proof}[Proof of Corollary \ref{cor:kahler-deforms-no-subvarieties}]
Since $Y$ contains no rational curves, it is not uniruled. Hence \cite[Theorem~1.1]{OuUniruledKahler} implies that $K_Y$ is pseudo-effective. Since $Y$ contains neither divisors nor surfaces,
Theorem~\ref{thm:thmmain} gives
\[
K_Y^{\otimes m}\simeq\mathcal O_Y
\]
for some $m>0$. Thus $\kappa(Y)=0$. Moreover, $Y$ contains neither divisors nor rational curves. By \cite[Theorem~2.8]{Vikash2026ClassificationOP},
\[
Y\simeq (T\times H)/G,
\]
where $T$ is a poor complex torus, $H$ is a product of poor irreducible holomorphic symplectic manifolds, and $G$ is finite and acts freely. The finite \'etale cover $T\times H\to Y$ also contains no proper positive-dimensional analytic subvarieties. Hence only one positive-dimensional factor can occur. Therefore either $T$ is a four-dimensional complex torus and $H$ is a point, or $T$ is a point and $H$ is an irreducible holomorphic symplectic fourfold. In the latter case, $G=\{1\}$ by \cite[Lemma~2.7]{Vikash2026ClassificationOP}, so $Y$ is irreducible holomorphic symplectic. In the former case, write
\[
g(x)=L_g(x)+a_g,\qquad g\in G.
\]
Since the action is free, $1-L_g$ is not surjective. Hence $\ker(1-L_g)^0$ is a positive-dimensional subtorus of $T$. Since $T$ has no proper positive-dimensional analytic subvarieties, it follows that $L_g=\operatorname{id}_T$. Thus $G$ acts by translations, and $Y=T/G$ is again a complex torus. Consequently, $Y$ is either a complex torus or an irreducible
holomorphic symplectic fourfold. By Ehresmann's theorem \cite{Ehresmann51}, $X$ and $Y$ are diffeomorphic. If $Y$ is a complex torus, then $X$ is a compact K\"ahler homotopy torus, and therefore is itself a complex torus by \cite{BauesCortes06}. If $Y$ is irreducible holomorphic symplectic, then so is $X$, since the irreducible holomorphic symplectic property is invariant under deformation \cite{Beauville83,Huybrechts99}. This proves the corollary.
\end{proof}

\subsection{Uniruled case}
The idea for the proof of this lemma is obtained with the help of ChatGPT \cite{ChatGPT1}
\begin{lemma} \label{lem:torus-positive-class}
    Let $T$ be a simple complex torus of dimension $n\geq 2$. If $\gamma\in H^{n-1,n-1}(T)\cap H^{2n-2}(T,\mathbb Q)$ satisfies $\gamma\cdot\omega\geq 0$ for every K\"ahler class $\omega\in\mathcal{K}(T)$, then $\gamma=0$.
\end{lemma}

\begin{proof}
    Write $T=V/\Lambda$, where $V$ is a real vector space of dimension $2n$ equipped with its complex structure $J$, and put $V_{\mathbb Q}=\Lambda\otimes_{\mathbb Z}\mathbb Q$. Choose a translation-invariant volume form $\mathrm{vol}\in\bigwedge^{2n}V_{\mathbb Q}^{*}$. The translation-invariant representative $\gamma_0$ of $\gamma$ determines a unique rational bivector $Q\in\bigwedge^2V_{\mathbb Q}$ by
    \begin{equation}\label{eq:torus-bivector}
    \iota_Q\mathrm{vol}=\gamma_0.
    \end{equation}
    Here $Q$ is $J$-invariant. Every K\"ahler class on $T$ has a positive definite translation-invariant representative, and every such representative defines a K\"ahler class. Since $\gamma \cdot \omega \geq 0$, $Q$ is a positive semidefinite Hermitian bivector. The image $W=\im(Q^{\sharp}:V^*\to V)$ is a rational $J$-invariant real subspace. Hence $W/(W\cap\Lambda)$ is a complex subtorus of $T$. Simplicity implies that either $Q=0$ or $Q^{\sharp}$ is invertible. Suppose that the latter alternative holds. Then $Q$ is positive definite. The translation-invariant two-form
    \begin{equation}\label{eq:rational-kahler-form}
    \eta:=\frac{1}{(n-1)!}\iota_{Q^{\wedge(n-1)}}\mathrm{vol}
    \end{equation}
    is rational and of type $(1,1)$. It is positive definite: choose a real basis $e_1,Je_1,\ldots,e_n,Je_n$ in which $Q=\sum_{j=1}^n a_j e_j\wedge Je_j$ with all $a_j>0$, and let $\theta_j=e_j^*\wedge(Je_j)^*$. Writing $\mathrm{vol}=c\,\theta_1\wedge\cdots\wedge\theta_n$, where $c>0$, formula \eqref{eq:rational-kahler-form} becomes
    \[
    \eta=c\sum_{j=1}^n\left(\prod_{k\ne j}a_k\right)\theta_j.
    \]
    Consequently $T$ admits a rational K\"ahler class. Clearing denominators and applying the Kodaira embedding theorem makes $T$ projective, a contradiction. Therefore $Q=0$, and \eqref{eq:torus-bivector} gives $\gamma=0$.
\end{proof}
\begin{remark}
    Sometimes simple means no nontrivial subtori, that is enough for the above proof if we assume non projective. There are examples of Abelian varieties which admit no non trivial sub tori. For us simple means, simplicity introduced by Campana, no non trivial subvarieties through a very general point. Obviously Campana simple implies no subtori.
\end{remark}

\begin{proof}[Proof of Proposition \ref{lem:p1-bundle-divisor}]
    Suppose for a contradiction that $Y$ contains no effective divisor. The obstruction to lifting its $\operatorname{PGL}_2(\mathcal O_T)$-cocycle to an $\operatorname{SL}_2(\mathcal O_T)$-cocycle lies in $H^2(T,\mu_2)$, by the central extension
    \[
    1\longrightarrow\mu_2\longrightarrow\operatorname{SL}_2(\mathcal O_T)
    \longrightarrow\operatorname{PGL}_2(\mathcal O_T)\longrightarrow1.
    \]
    The multiplication map $[2]:T\to T$ acts on $H^2(T,\mu_2)\simeq H^2(T,\mathbb Z/2)$ as multiplication by $4$, hence by zero. After this finite \mbox{\'etale} base change, we therefore have
    \begin{equation}\label{eq:sl2-lift}
    Y':=Y\times_{T,[2]}T\simeq\mathbb P(E),\qquad
    \rk E=2,\qquad\det E\simeq\mathcal O_T.
    \end{equation}
    $Y'$ also contains no effective divisor. Denote the projection in \eqref{eq:sl2-lift} by $p':\mathbb P(E)\to T$. The thechnique in the following proof I learned from \cite{Campana2026Bogomolov}.
    \begin{claim}
        $E$ has no rank-one coherent subsheaf and therefore slope-stable with respect to every K\"ahler class on $T$.
    \end{claim}
    \begin{proof}
         Suppose for a contradiction that $\mathscr F\subset E$ is torsion-free of rank one, then its double dual $L=\mathscr F^{**}$ is a line bundle. The inclusion gives a nonzero morphism $L\to E$ away from a subset of codimension at least two, and this morphism extends over $T$ by Hartogs extension for locally free sheaves. Its composition with the tautological quotient $(p')^*E\to\mathcal O_{\mathbb P(E)}(1)$ gives a nonzero section of
        \[
        \mathcal O_{\mathbb P(E)}(1)\otimes(p')^*L^{-1}.
        \]
        On a general fiber $\mathbb P^1$ this section has one zero. Its zero locus is therefore a nonempty effective divisor in $\mathbb P(E)$, a contradiction.
    \end{proof}
    By \eqref{eq:sl2-lift}, its first Chern class vanishes. The Bogomolov-L\"ubke inequality thus gives
    \begin{equation}\label{eq:c2-all-kahler}
    \int_T c_2(E)\wedge\omega\geq0
    \qquad\text{for every }\omega\in\mathcal K(T).
    \end{equation}
    Because $\dim T=3$, the class $c_2(E)$ has type $(n-1,n-1)$. Lemma~\ref{lem:torus-positive-class}, applied to \eqref{eq:c2-all-kahler}, yields $c_2(E)=0$ in $H^4(T,\mathbb Q)$. Fix a K\"ahler class $\omega$. By the Hermitian-Einstein theorem \cite{UY86}, the $\omega$-stable bundle $E$ admits a Hermitian-Einstein connection. The equalities $c_1(E)=0$ and $c_2(E)\cdot\omega=0$, imply that this connection is flat. Its holonomy is unitary. Since $\pi_1(T)\simeq\mathbb Z^6$ is abelian, its unitary representation decomposes into one-dimensional characters. Hence $E$ is a direct sum of flat line bundles, contradicting the absence of rank-one coherent subsheaves.
\end{proof}
\begin{proof}[Proof of Theorem \ref{thm:thmuniruled}]
    Suppose $X$ is uniruled, then take the MRC quotient $r: X\dashrightarrow B$, and choose $B$ to be smooth; see \cite[Theorem~2.6]{Campana04}. Due to absence of effective divisors and surfaces we have 
    \begin{equation}\label{eq:mrc-dimension}
    \dim B=3,\qquad\text{the general fiber of }r\text{ is }\mathbb P^1.
    \end{equation}
    Observe that the threefold $B$ is simple. By \cite[Corollary~1.4]{DasOu26}, a simple compact K\"ahler threefold with klt singularities is bimeromorphic to a finite quotient of a complex torus. Hence we obtain a map $B\dashrightarrow T/G$, where $T$ is a complex torus and $G$ acts faithfully on it. Observe that $T$ has no proper subvarieties, because translates of them will cover $T$ and therefore $T/G$ contradicting the simplicity of $B$. Let $q: X \dashrightarrow T/G$ be the resulting meromorphic map. Resolving $q$ we get,
    \[
    \begin{tikzcd}
    \widetilde X \arrow[r,"\mu"] \arrow[dr,"\widetilde q"']&
    X \arrow[d,dashed,"q"]\\
    &T/G.
    \end{tikzcd}
    \]
    Fibers of $\mu$ map to points in $T/G$, because $T/G$ contains no curves. Hence $q$ descends to to a holomorphic map and call it $q$ again. Now pull back $q$ via the the quotient map $T \to T/G$ and normalize and consider the component $Y$ that dominates $T$. We have the square 
    \[
    \begin{tikzcd}
    Y \arrow[r,"p"] \arrow[d, "\widetilde q"] & X \arrow[d,"q"]\\
    T \arrow[r,"\pi"] & T/G .
    \end{tikzcd}
    \]
    Because $\pi$ is finite $p$ is finite. Since $Y$ is normal and $X$ is smooth, the branch locus of $p$ is empty or a codimension one subvariety in $X$, therefore empty. Hence $p$ is finite \'etale and hence $Y$ is smooth and has no subvarieties of codimension one and two. The map $\widetilde q$ is surjective holomorphic and $Y$ has no subvarieties of codimension one and two, so every fiber of $\widetilde q$ is one dimensional and a general fiber is $\mathbb P^1$. Then by \cite[Proposition 13.2]{BandmanZarhin23Bundles}, $\widetilde q$ is $\mathbb P^1$- bundle, therefore a contradiction due to Proposition \ref{lem:p1-bundle-divisor}.
\end{proof}

\begin{proof}[Proof of corollary \ref{cor:final}]
    Follows from Theorem \ref{thm:thmmain}, Theorem \ref{thm:thmuniruled}, \cite{OuUniruledKahler} and \cite{Beauville83}. Following the argument in \cite[Theorem 2.8]{Vikash2026ClassificationOP}.
\end{proof}
\section{Examples}\label{ex:gen}
The notion of poor manifolds was introduced by Bandman and Zarhin \cite{bandman2023simple}, and was studied in \cite{bandman2023simple,Vikash2026ClassificationOP}. Explicit examples of tori without subvarieties were also constructed in \cite{bandman2023simple}. Examples of Irreducible holomorphic symplectic manifolds without subvarieties were constructed in \cite{Verbitsky1998}. The purpose of this section is to construct examples of four dimensional K\"ahler manifolds without surfaces and codimension one subvarieties inside it and has Euler characteristic zero.

\begin{example}
    We construct a four dimensional torus $T$, admitting a smooth elliptic fibration 
    \[
    f: T \to B
    \]
    such that $T$ has no subvarities of codimension one and two and  $h^0(X,\Omega_X^3)=4$. By Bandman-Zarhin \cite[Definition~1.2, Theorem~1.3, Proposition~1.7]{bandman2023simple}, there exists a three-dimensional complex torus
    \[
    B= (W,J_B)/\Lambda_B
    \]
    Which has no proper subvarieties of positive dimension and satisfies $\operatorname{NS}(B)=0$. We now construct elliptic fibrations over $B$. Fix an elliptic curve
    \[
    E=(U,J_E)/\Lambda_E,
    \]
    where $U$ is a real two dimensional vector space, $J_E^2=-1$, and $\Lambda_E \subset U$ is a lattice. Put
    \[
    V:= U\oplus W, \quad \Lambda:=\Lambda_E \oplus \Lambda_B.
    \]
    We want a complex structure on the the fixed real torus $V/\Lambda$ in such a way that the projection to $B$ remains holomorphic and kernel is fixed elliptic curve $E$. Explicitly a complex structure such that the its restriction to $U$ part must be $J_E$ and its projection should be $J_B$. which explains the following definition.
    \[
    \mathcal A:=
    \{A\in \operatorname{Hom}_{\mathbb R}(W,U)\mid J_EA+AJ_B=0\}.
    \]
    Now for $A\in \mathcal{A}$, define 
    \[
    J_A(u,w):=(J_Eu+Aw,J_Bw), \qquad (u,w)\in U\oplus W.
    \]
    Observe that due to the definition of $\mathcal{A}$, we have 
    \[
    J^2_A(u,w)=(-u +(J_EA+AJ_B)w,-w)=(-u,-w).
    \]
    Therefore $T_A:= (V,J_A)/\Lambda$ is a four dimensional complex torus. Observe that the projection 
    \[
    p: V= U\oplus W \to W
    \]
    is complex linear from $(V,J_A)$ to $(W,J_B)$ because $p(J_A(u,w))=J_Bp(u,w)$. Moreover $p(\Lambda)=\Lambda_B$. Hence $p$ gives a surjective homomorphism $f_{T_A}: T_A \to B$, its kernel is $\operatorname{ker}(f_{T_A}) = U/\Lambda_{E}=E$. In particular $f_{T_A}$ is smooth isotrivial fibration. It remains to choose $A\in \mathcal{A}$ such that $\operatorname{NS}(T_A)=0$.
    \begin{claim}\label{claim:varry}
        The following hold.
        \begin{itemize}
            \item[$1$] If $w\in W$ is nonzero, then given any $u\in U$ there exists $A\in \mathcal{A}$ such that $Aw=u$.
            \item[$2$] If $w_1,w_2\in W$ are nonzero linearly independent wrt $\mathbb{C}$, then given any $u\in U$ there exists $A\in \mathcal{A}$ such that $Aw_1=u$ and $Aw_2=0$.
        \end{itemize}
        \begin{proof}
            Let $w\in W$ be nonzero and let $u\in U$. Define $Aw:=u$, $A(J_Bw)=-J_E(u)$, extend basis and set every thing else zero. Therefore proves $1$. For the second statement, set 
            \[
            A(w_1):= u, \quad A(J_Bw_1)=-J_Eu, \quad A(w_2):= 0,\quad A(J_Bw_2)=0
            \]
            extend the basis and set everything zero proves $2$.
        \end{proof}
    \end{claim}
    We now use the Appell-Humbert description of line bundles on complex tori. If $T= V/\Lambda$ is a complex torus then $\operatorname{NS}(T)$ is given by elements $Q\in \bigwedge^2\Lambda^{\vee}$, satisfying the Hodge condition $Q(Jx,Jy)=Q(x,y)$ for all $x,y\in V$, see \cite[Chapter~2]{BL04}. This description motivates the following definition,
    \[
    S_Q:=\{A\in \mathcal{A} \mid Q(J_Ax,J_Ay) = Q(x,y)\}.
    \]
    It is easy to see that $S_Q$ is a real algebraic space inside $\mathcal{A}$, and has empty interior. Indeed, $S_Q$ is defined by vanising of polynomials after we choose a basis of $V$.
    \begin{claim}
        For every $0\neq Q\in \bigwedge^2\Lambda^{\vee}$, we have $S_Q\subsetneq \mathcal{A}$.
    \end{claim}
    \begin{proof}
        Suppose for a contradiction that $S_Q=\mathcal{A}$. Decompose $Q$ using $V=U\oplus W$ as $Q=Q_{UU}+Q_{UW}+Q_{WW}$, where
        \[
        Q_{UU}\in \bigwedge^2U^{\vee}, \quad Q_{UW}\in U^{\vee}\otimes W^{\vee}, \quad Q_{WW}\in \bigwedge^2W^{\vee}.
        \]
        Therefore we have
        \[
        \begin{aligned}
        Q((u_1,w_1),(u_2,w_2))
        &=
        Q_{UU}(u_1,u_2)
        +Q_{UW}(u_1,w_2)  \\
        &\quad
        -Q_{UW}(u_2,w_1)
        +Q_{WW}(w_1,w_2).
        \end{aligned}
        \]
        where $w_i\in W; u_i\in U$.  In the next formula we regard $u$ as $(u,0)\in V$, and $w$ as $(0,w)\in V$. Then due to the assumption $S_Q=\mathcal{A}$ we have
        \[
        Q((u,0),(0,w))= Q(J_A(u,0),J_A(0,w))= Q((J_Eu,0),(Aw,J_Bw)),
        \]
        using decomposition we get 
        \[
        Q_{UU}(J_Eu,Aw)+Q_{UW}(J_Eu,J_Bw)=Q_{UW}(u,w).
        \]
        Because of claim \ref{claim:varry} and $S_Q=\mathcal{A}$, we can change $A$ so that $Aw$ spans $U$, hence $Q_{UU}(J_Eu,u')=0, \forall u'\in U$, since $J_E$ is automorphism, we have $Q_{UU}=0$. Now, Again due to the assumption $S_Q=\mathcal{A}$, we have $Q(J_A(0,w_1),J_A(0,w_2))=Q((0,w_1),(0,w_2))$. Using the decomposition and $Q_{UU}=0$, we have 
        \[
        \begin{aligned}
        &Q_{UW}(Aw_1,J_Bw_2)
        -
        Q_{UW}(Aw_2,J_Bw_1) \\
        &\qquad
        +
        Q_{WW}(J_Bw_1,J_Bw_2) = Q_{WW}(w_1,w_2),
        \end{aligned}
        \]
        where $w_1,w_2 \in W$ are linearly independent wrt $\mathbb{C}$. By claim \ref{claim:varry}, for every $u\in U$, there is $A\in \mathcal{A}$ such that $Aw_1=u$ and $Aw_2=0$. Therefore varying $u$ we get $Q_{UW}(u,J_Bw_2)=0$ for all $u\in U$, hence $Q_{UW}=0$. Then $Q(w_1,w_2)=Q(J_Aw_1,J_Aw_2)=Q(J_Bw_1,J_Bw_2)$ for all $w_1,w_2\in W$. Finally if $Q_{WW}\neq 0$ this defines a class in $\operatorname{NS}(B)$, contradicting the assumption of the construction that $NS(B)=0$. Therefore $Q=0$, contradicting the assumption of the claim. Hence $S_Q\neq \mathcal{A}$. 
    \end{proof}
    Now observe that $S_Q$ is closed and has empty interior in $\mathcal{A}$. By the Baire category theorem we can choose
    \[
    A\in
    \mathcal A\setminus
    \bigcup_{0\neq Q\in\bigwedge^2\Lambda^\vee}S_Q.
    \]
    Therefore $T_A$ admits an elliptic fibration and $\operatorname{NS}(T_A)=0$. $T_A$ has no divisors because $\operatorname{NS}(T_A)=0$ and no surfaces because they would give a proper subvariety in $B$, which is not possible. In paricular 
    \[
     h^0(T_A,\Omega_{T_A}^3)=\binom{4}{3}=4,
    \qquad
    K_{T_A}\simeq\mathcal O_{T_A},
    \qquad
    \chi(T_A,\mathcal O_{T_A})=0.
    \]
\end{example}

\section*{Acknowledgements}

I am grateful to my advisors Tatiana Bandman, John Lesieutre, and Yuriy Zarhin for many helpful discussions and for their guidance. I am grateful to Daniel Huybrechts, Rahim Moosa, Anand Pillay for pointing references related to Model theory. I am also grateful to Thomas Peternell and Frédéric Campana for helpful comments. I would like to thank several graduate students at Penn State, including Andy B. Day, Eugene Henninger-Voss, Neelarnab Raha, Satwata Hans, Xingkai Wang, and Louis Diaz, for helpful discussions.

\section*{AI Disclosure and Usage}
During the preparation of this manuscript, the author used ChatGPT 5.6 at two places. It helped the author find a proof of Lemma~\ref{lem:ext}, a standard statement for which a reference in the precise form needed here was difficult to locate. After obtaining eq \ref{eq:c2-all-kahler}, the author prompted for ideas to prove $c_2=0$, which eventually led to the proof of Lemma \ref{lem:torus-positive-class}. Author takes full responsibility for all the arguments written.

\bigskip

\noindent\textsc{Pisya Vikash}\\
Department of Mathematics\\
The Pennsylvania State University\\
\textit{Email address}: \texttt{pmv5172@psu.edu}.

\end{document}